\documentclass[11pt]{amsart}
\usepackage{amssymb,latexsym,amsmath}
\numberwithin{equation}{section}
\newtheorem{definition}[equation]{Definition}
\newtheorem{theorem}[equation]{Theorem}
\newtheorem{proposition}[equation]{Proposition}
\newtheorem{corollary}[equation]{Corollary}

\newtheorem{example}[equation]{Example}
\newtheorem{remark}[equation]{Remark}
\newtheorem{conjecture}[equation]{Conjecture}

 \newtheoremstyle{exercises}
  {3pt}
  {6pt}
  {}
  {}
  {\bfseries}
  {:}
  {\newline}
   {}

\theoremstyle{exercises}
\newtheorem{exercises}[equation]{Exercises}

\newcommand{\exerset}[2][{}]{
\begin{exercises}{#1}
\exersetmiddle{#2}
\end{exercises}
}

\newcommand{\exersetmiddle}[1]
{
\begin{list}{\arabic{enumi}.}
{\usecounter{enumi}\exerfuss}
{#1}\end{list}
\vskip 10pt
}
\newcommand{\exerfuss}{
\setlength{\itemsep}{0pt}
\setlength{\topsep}{-10pt} 
\setlength{\leftmargin}{0pt}  
\setlength{\labelwidth}{1em}
\setlength{\labelsep}{0.6em}
\setlength{\itemindent}{1.6em}
}

\newcommand{\Cbox}[2]{\parbox{#1}{\begin{center}{#2}\end{center}}}
 
\def\a{{\alpha}}
\def\b{\beta}

\def\d{\delta}
\def\e{\varepsilon}

\def\m{\mu}

\def\w{\omega}

\def\th{\theta}

\def\oa{{\overline \a}}

\def\bP{\mathbb{P}}

\def\bZ{\mathbb{Z}}

\def\cF{\mathcal{F}}

\def\ff{{\mathfrak{f}}}
\def\fg{{\mathfrak{g}}}
\def\fgl{\mathfrak{gl}}

\def\fp{\mathfrak{p}}

\def\fsl{\mathfrak{sl}_{m+1}}
\def\fso{\mathfrak{so}}
\def\fsp{\mathfrak{sp}_{2m}}
\def\ft{\mathfrak{t}}

\def\td{\mathrm{d}}
\def\tdim{\mathrm{dim}}

\def\tFub{\mathrm{Fub}}

\def\tker{\mathrm{ker}}
\def\tmod{\textrm{ mod }}
\def\tId{\textrm{Id}}

\def\tspan{\mathrm{span}}

\def\l{\ell}
\def\op{\oplus}
\def\ot{\otimes}
\def\t{\tilde}

\def\t{\tau}
\def\s{\sigma}

\def\frp#1#2{\frac{\partial{#1}}{\partial{#2}}}

\def\ooo#1#2{\omega^{#1}_{#2}}
\def\oo#1{\omega^{#1}_0}
\def\ep{\epsilon} 
\def\qq#1#2#3{q^{#1}_{{#2} {#3}}}
\def\rr#1#2#3#4{r^{#1}_{{#2} {#3}{#4}}}

\def\cf{\mathcal F}
\def\bcc#1{\BC^{#1}}

\def\trank{\text{rank}}

\def\BC{\mathbb C}\def\BF{\mathbb F}\def\BO{\mathbb O}\def\BS{\mathbb S}\def\BE{\mathbb E}
\def\BR{\mathbb R}

\def\BP{\mathbb P}
\def\pp#1{\mathbb P^{#1}}
\def\fa{\mathfrak a}\def\fr{\mathfrak r}
\def\fgl{\mathfrak g\mathfrak l}

\def\fc{\mathfrak c}

\def\ppp{{\mathbb P}}
\def\pp#1{{\mathbb P}^{#1}}
\def\tdim{\rm dim}
\def\hd{,...,}
\def\ww{\wedge}
\def\upperp{{}^\perp}

\def\inv{{}^{-1}}

\def\cJ{{\mathcal J}}

\def\cF{{\mathcal F}}

\def\cV{{\mathcal V}}

\def\CC{\mathbb C}

\def\BZ{\mathbb Z}

\def\11{\mathbf 1}
\def\PP{\mathbb P}
\def\QQ{\mathbb Q}

\def\fsl{{\mathfrak {sl}}}
\def\fsp{{\mathfrak {sp}}}

\def\fso{{\mathfrak {so}}}
\def\fe{{\mathfrak e}}

\def\ff{{\mathfrak f}}

\def\fg{{\mathfrak g}}
\def\fn{{\mathfrak n}}
\def\fp{{\mathfrak p}}

\def\ft{{\mathfrak t}}
\def\fl{{\mathfrak l}}
\def\l{\lambda}
\def\a{\alpha}

\def\o{\omega}

\def\O{\Omega}
\def\b{\beta}
\def\g{\gamma}
\def\s{\sigma}

\def\d{\delta}
\def\th{\theta}
\def\m{\mu}
\def\up#1{{}^{({#1})}}
\def\e{\varepsilon}
\def\ot{{\mathord{\,\otimes }\,}}
\def\op{{\mathord{\,\oplus }\,}}

\def\lra{{\mathord{\;\longrightarrow\;}}}
\def\ra{{\mathord{\;\rightarrow\;}}}

\def\we{{\mathord{{\scriptstyle \wedge}}}}

\def\La#1{\Lambda^{#1}}
\def\tim{\text{Image}\,}

\def\tdim{\text{dim}\,}
\def\tcodim{\text{codim}\,}

\def\tker{\text{ker}\,}
\def\tspan{\text{span}\,}
\def\tmod{\text{ mod }}

\def\trank{\text{rank}\,}

\def\be{\begin{equation}}
\def\ene{\end{equation}}

\def\lra{\ \Longrightarrow\ }

\def\gp#1{\fg^{\perp}_{#1}}

\begin{document}
\title{Exterior differential systems, Lie algebra cohomology, and the rigidity of homogenous varieties}
\author{J.M. Landsberg}
\date{\today}
\begin{abstract}These are expository notes from the 2008 Srni Winter School.
They have two purposes: (1)  to give a quick introduction
to exterior differential systems (EDS), which is a collection of techniques
for determining local existence to systems of partial differential equations,
and (2) to give an exposition of recent work (joint with C. Robles)
on the study of the Fubini-Griffiths-Harris rigidity of rational homogeneous
varieties, which also involves an advance in the EDS technology.
\end{abstract}
\maketitle
 
\renewcommand{\thefootnote}{\arabic{footnote}}

\tableofcontents

\section{Introduction}\label{intro}

Let $G$ be a complex semi-simple Lie (or algebraic) group, and let 
$V=V_{\l}$ be an irreducible $G$-module.
The homogeneous variety $G/P= G.[v_{\l}]\subset \BP V$  
is the orbit of a highest weight line.

For example, let $W$ be a complex vector space, $V=\La k W$ and
let $G=SL(W)$,
then $G/P=G(k,W)\subset \BP(\La k W)$ is the {\it Grassmannian} of
$k$-planes through the origin in $W$ in its Plucker embedding.

A long term program with my collaborators Laurent Manivel, Colleen Robles
and Jerzy Weyman is to study
  relations  between the   projective geometry of $G/P\subset \BP V$,
especially its local differential geometry, and the representation theory
of $G$. More than just the geometry of $G/P$, we are interested in the
geometry of its auxiliary varieties, for example the {\it tangential variety}
$\t(G/P)\subset \BP V$, which is the union of  all points on all embedded
tangent lines to $G/P$, and the {\it $r$-th secant variety} of $G/P$,
$\s_r(G/P)\subset \BP V$, which is the Zariski closure of
all points on all secant $\pp{r-1}$'s to $G/P$.
The auxilary varieties are all $G$-varieties, i.e., preserved
under the action of $G$, and thus one can study their ideals, coordinate
rings, etc.  as $G$-modules.

\subsection{Overview}
These notes are focused on the local projective differential geometry of
homogeneously embedded rational homogeneous varieties $G/P\subset \BP V$.
Specifically,
they address the question
how much of the local geometry is needed to recover $G/P$.
We begin by describing many examples of rational homogeneous varieties in
\S\ref{ratexsect}. The main question we deal with is rigidity,
but before discussing rigidity questions, we give 
  descriptions of related projects   in \S\ref{otherprojects} to  give   context to
this work.
  The rigidity results and questions are described in
\S\ref{projgeomsect}. In \S\ref{pdeedssect} and \S\ref{cartanalgsect} we give a crash
course on {\it exterior differential systems} (EDS).
Roughly speaking, EDS is a collection of techniques for determining the space of
local solutions to systems of partial differential equations. The techniques usually
involve extensive computations that can be simplified by exploiting group actions
when such are present, as with the rigidity questions that will be the focus of
this paper.
 In \S\ref{projframessect} we describe
moving frames for submanifolds of projective space and a set of   \lq\lq rigidity\rq\rq\  EDS that
are natural from the point of view of projective differential geometry.
We also describe flexibility results obtained using standard EDS techniques.
A   different method for resolving certain EDS associated to
determining the rigidity of compact Hermitian symmetric spaces (CHSS)
 was introduced
by Hwang and Yamaguchi in \cite{HY} that avoided lengthy calculations
by reducing the proof    to establishing the  vanishing of certain Lie algebra
cohomology groups. At first, it appeared that their methods would not
extend beyond the CHSS cases, but the machinery was finally extended
in \cite{LRrigid}. {\it This extension is the central point of these lectures}.  Several problems had to be overcome to
enable the extension - the problems and their solutions are
discussed in detail in in \S\ref{filtersect}. The first problem
is that the EDS natural for geometry is not natural for representation
theory, once one moves beyond CHSS. This problem is (partially) resolved
in \S\ref{problem1} with the introduction of the systems $(I_p,J_p)$ which are natural for representation
theory.
The next problem is that even these natural systems do not lead one to
Lie algebra cohomology, except in the case of CHSS. However a refined version
of the $(I_p,J_p)$ systems, the {\it filtered systems} $(I^f_p,\Omega)$ do. 
This is explained in \S\ref{problem2}, which then leads to our main theorem, Theorem
\ref{mainthm}.
Before discussing these systems, we describe and compare, for $G/P\subset \BP V$  
the  filtration of $V$ induced by the osculating sequence and a filtration
induced by the Lie algebra in \S\ref{gradingsect}, and briefly review Lie algebra cohomology 
in  \S\ref{Liealgcohsect}.

\subsection{Examples of rational homogeneous varieties}\label{ratexsect}
\subsubsection{Generalized cominuscule varieties}
The simplest rational homogeneous varieties   are
the {\it generalized cominuscule varieties}, which are the homogeneously
embedded compact Hermitian symmetric spaces. In addition to the 
Grassmannians mentioned above, the {\it cominuscule varieties}, which are the irreducible 
CHSS in their minimal homogeneous embeddings,  are
\begin{itemize}
\item the {\it Lagrangian Grassmannians} $G_{\o}(n,W)=C_n/P_n\subset \BP (\La n W/\o \ww \La{n-2}W)$, where $W$
is a $2n$ dimensional vector space equipped with a symplectic form 
$\o\in \La 2 W^*$, $C_n$ is the group preserving the form, and $G_{\o}(n,W)\subset G(n,W)$
are the $n$-planes on which $\o$ restricts to be zero. (Note that we may
use $\o$ to identify $W$ with $W^*$ so $\o \ww \La{n-2}W$ makes sense.)
\item  the {\it Spinor varieties} $\BS_n=D_n/P_n\simeq D_n/P_{n-1}$ which are also isotropic Grassmannians, only
for a symmetric quadratic form, where $W$ again has dimension $2n$. Their minimal homogeneous embedding
is in a space smaller than $\BP (\La nW)$.
\item  the {\it quadric hypersurfaces} $Q^{n-1}=G_Q(1,W)\subset \BP W$, (which
are $B_m/P_1$ and $D_m/P_1$ depending if $n=2m+1$ or $n=2m$)
\item  the {\it Cayley plane} $\BO\pp 2=E_6/P_6\simeq E_6/P_1\subset\BP \cJ_3(\BO)$ which are  
the octonionic lines in $\BO^2$ embedded as the rank one elements
of the exceptional Jordan algebra $\cJ_3(\BO)$, see, e.g., \cite{LMmagic} for details. 
\item  the {\it Freudenthal
variety} $E_7/P_7\subset \pp{55}$ which may be thought of as
an octonionic Lagrangian Grassmanian $G_w(\BO^3,\BO^6)$, see \cite{LMmagic}.
\end{itemize}

\subsubsection{Products of homogeneous varieties}
An elementary, but important generalized cominuscule variety is the {\it Segre variety}. Let
$V,W$ be vector spaces, the Segre variety $Seg(\BP V\times \BP W) \subset \BP (V\ot W)$ as
an abstract variety is simply the product of two projective spaces. It
is embedded as the set of rank one elements of $V\ot W$. 
  In general, if $G/P\subset \BP V$ and $G'/P'\subset \BP V'$,
we may form the product $Seg(G/P\times G'/P')\subset \BP  (V\ot V')$, which
is of course a subvariety of $Seg(\BP V\times \BP V')$.

\subsubsection{Veronese re-embeddings of homogeneous varieties}
Considering $S^dV$ as the space of homogeneous polynomials of degree $d$ on
$V^*$, we can consider the variety of $d$-th powers inside $\BP (S^dV)$,
this is isomorphic to $\BP V$ via the map $v_d: \BP V\ra \BP S^dV$,
$[x]\mapsto [x^d]$, called the {\it Veronese embedding}.
If $X\subset \BP V$ is a subvariety we can consider $v_d(X)$.
Its linear span $\langle v_d(X)\rangle\subset S^dV$   has the geometric interpretation
of the annihilator of  $I_d(X)\subset S^dV^*$, the ideal of $X$ in degree
$d$. In particular, if $X=G/P\subset \BP V_{\l}$ is homogeneous, then
$\langle v_d(G/P)\rangle= V_{d\l}$, the $d$-th {\it Cartan power}
of $V_{\l}$.

\subsubsection{Generalized flag varieties}  Given two   Grassmannians
$G(k,V)$ and $G(\ell,V)$ with say $k<\ell$, we may form the incidence
variety $Flag_{k,\ell}(V)=\{ (E,F)\in G(k,V)\times G(\ell,V)\mid E\subset F\}$.
Of course $Flag_{k,\ell}(V)\subset \BP (\La k V\ot \La {\ell}V)$
Write $\La k V= V_{\o_k}$. Then in fact
$\langle Flag_{k,\ell}(V)\rangle =V_{\o_k+\o_{\ell}}$ giving a geometric
realization of the Cartan product of the two modules $V_{\o_k}$ and $ V_{\o_{\ell}}$.
This generalizes to arbitrary Cartan products as follows:

The cominuscule varieties
are special cases of \lq\lq generalized Grassmannians\rq\rq , that is varieties
$G/P$ where 
$P$ is a maximal parabolic. Such varieties always admit interpretations as
subvarieties of some Grassmannian, usually given in terms of the set of $k$ planes
annihilated  by some tensor(s).
 Given two such for the same group,
$G/P_i\subset \BP V_{\o_i}$ and $G/P_j\subset \BP V_{\o_j}$, we may form an incidence variety $G/P_{i,j}$
and again we will have $\langle G/P_{i,j}\rangle= V_{\o_i+\o_j}$.
Thus  {\it  Cartan powers and products of modules can be constructed geometrically}.

\subsubsection{Adjoint varieties}
After the generalized cominuscule varieties, the next simplest rational homogeneous varieties are the {\it adjoint varieties},
where $V$ is taken to be $\fg$, the adjoint representation of $G$.
We write $G/P=X^{ad}_G\subset\BP \fg$ to denote adjoint varieties.
The adjoint varieties can also be characterized as the homogeneous
compact complex contact manifolds. It is conjectured (see, e.g.,
\cite{LeBrunSalamon,Kebetal}) that they are essentially the only   compact complex contact
manifolds other then projectivized cotangent bundles.
Many of these have simple geometric interpretations.
\begin{itemize}
\item 
$X^{ad}_{SL(W)}=Flag_{1,n-1}(W)$ is the variety of flags of lines
in hyperplanes in the $n$-dimensional vector space $W$.
\item 
$X^{ad}_{SO(W,Q)}=G_Q(2,W)\subset \BP (\La 2 W)=\BP \fso(W)$ is the Grassmannian of isotropic
$2$-planes in $W$. 
\item $X^{ad}_{G_2}=G_{null}(2,Im\BO)$ is  the Grassmannian of
two planes in the imaginary octonions on which the multiplication is zero.
It may also be seen as the projectivization of the set of rank two derivations
of $\BO$, or as the set of six dimensional subalgebras of $\BO$, see
\cite{LMsex}, Theorem 3.1.
\item
  $X^{ad}_{Sp(W,\o)}=v_2(\BP W)\subset \BP S^2W=\BP \fc_n$ is the variety of quadratic
forms of rank one.
\end{itemize}
Note that other than the pathological groups $A_n,C_n$, all adjoint representations
are fundamental. Also note that the adjoint variety of $\fc_n$ is generalized
cominuscule for $\fa_{2n-1}$.

\medskip

\subsection{Notational conventions}
We work over the complex numbers throughout, all functions are holomorphic functions 
  and manifolds are complex manifolds  (although much of the theory carries over to $\BR$,
with some rigidity results even carrying over to the $C^{\infty}$ setting). In particular
the notion of a {\it general point} of an analytic manifold makes sense, which is a point
off of a finite union of analytic subvarieties.
We use the labeling and ordering of roots and
weights as in \cite{bou}. For subsets $X\subset \BP V$, $\hat X\subset V$ denotes the corresponding cone.  For a manifold
$X$, $T_xX$ denotes its tangent space at $x$. For a submanifold $X\subset \BP V$, $\hat T_xX = T_p\hat X\subset V$, 
denotes its affine tangent
space, and $p \in \hat x =: L_x$. In particular, $T_xX=\hat x^*\otimes \hat T_xX/\hat x$.
If $Y\subset \BP W$, then $\langle Y\rangle\subset W$ denotes its linear span. We use the summation
convention throughout: indices occurring
up and down are to be summed over. If $G$ is semi-simple of rank $r$, we write
  $P=P_I\subset G$ for the parabolic subgroup obtained by deleting negative root spaces corresponding to roots having a nonzero coefficient on any of the simple roots $\a_{i_s}$, $i_s \in I\subset\{1\hd r\}$.

\subsection{Acknowledgements} It is a pleasure to thank the organizers
of the 2008 Srni Winter School, especially A. Cap and J. Slovak. I also thank
C. Robles for useful suggestions.

\section{Related projects}\label{otherprojects}
\subsection{Representation theory and computational complexity}
 These projects with Manivel and   Weyman address questions about $G$-varieties motivated
by problems in computer science and algebraic statistics, specifically the complexity
of matrix multiplication and the study of phylogenetic invariants.
For a survey on this work see \cite{Lbull}.

\subsection{Sphericality and tangential varieties}\label{sphere}
For work related to Joachim Hilgert's lectures \cite{Hilgert}, recall that a
normal projective $G$-variety
$Z$ is {\it $G$-spherical} if   for all degrees $d$, $\BC [Z]_d$,
the component of the coordinate ring of $Z$ in degree $d$, is a multiplicity
free $G$-module, see \cite{brion}.
Note that this property for $Z=\t(X)$ {\it a priori} depends both
on $G$ and the embedding of $X$.

\begin{theorem}\label{spherethm}\cite{LWtan}
Let $X=G/P \subset \BP V$ be a 
homogeneously embedded rational homogeneous variety.
Then $\t (X)$ is $G$-spherical iff $X$ admits
the structure of a CHSS,
and no factor of $X$ is $G_2/P_1$. 
\end{theorem}

In \cite{LWtan} we also show that if  $G/P$ is
cominuscule then $\t(G/P)$ is normal, with rational
singularities, and   give explicit and uniform
descriptions of the coordinate rings for all cases in the spirit
of the project described in \S\ref{vogelia} below.

An interesting class of $\t(G/P)$'s occurs for the {\it subexceptional
series}, the third row of Freudenthal's magic chart:
$Seg(\pp 1\times \pp 1\times \pp 1)$,
$G_{\o}(3,6)$, $G(3,6)$, $D_6/P_6=\BS_6$, $E_7/P_7=G_w(\BO^3,\BO^6)$
where   $\t(G/P)$ is a quartic hypersurface whose equation is
given by a generalized hyperdeterminant. See \cite{LMmagic} for details.
The equations of these varieties will play an important role in what follows,
as the Fubini quartic forms  for $X^{ad}_G$  when $G$ is an exceptional group
(see \S\ref{fubsect}).

\subsection{Vogelia}\label{vogelia}
 This project, joint with Manivel, is
inspired conjectural categorical generalizations
of Lie algebras proposed by P. Deligne (for the exceptional series)
\cite{del,del2}
and P. Vogel (for all simple super Lie algebras) \cite{vogel}.
It  has relations Pierre Loday's lectures \cite{Loday} because
both conjectures appear to inspired by operads.

Let $\fg$ be a complex simple Lie algebra. Vogel  derived 
a universal decomposition of $S^2\fg$ into   (possibly virtual)
Casimir eigenspaces, $S^2\fg = \BC\op Y_2\op Y_2'\op Y_2''$
which turns out to be a decomposition into
irreducible modules.  If we let $2t$ denote
the Casimir eigenvalue of the adjoint representation (with respect to some 
invariant quadratic form), these modules
respectively have Casimir eigenvalues $4t-2\a,4t-2\b,4t-2\g$, which
we may take as the definitions of $\a,\b,\g$.
Vogel showed   that $t=\a+\b+\g$. For example,
for $\fso(n)$ we may take $(\a,\b,\g)=(-2,4,n-4)$ and
for the exceptional series $\fso_8,\ff_4,\fe_6,\fe_7,\fe_8$
we may take $(\a,\b,\g)=(-2,m+4,2m+4)$ where
$m=0,1,2,4,8$ respectively.
Vogel then went on
to find             Casimir eigenspaces $Y_3,Y_3',Y_3''\subset S^3\fg$
with eigenvalues $6t-6\a,6t-6\b,6t-6\g$
(which again turn out to be irreducible), and computed 
their dimensions:

\begin{align*}{\rm dim}\, \fg  &=  \frac{(\a-2t)(\b-2t)(\g-2t)}{\a\b\g},\\
 {\rm dim}\, Y_2  &=  -\frac{t(\b-2t)(\g-2t)(\b+t)(\g+t)(3\a-2t)}{\a^2\b\g(\a-\b)(\a-\g)}\\
 {\rm dim}\, Y_3  &=  -\frac{t(\a-2t)(\b-2t)(\g-2t)(\b+t)(\g+t)(t+\b-\a)(t+\g-\a)(5\a-2t)}
{\a^3\b\g(\a-\b)(\a-\g)(2\a-\b)(2\a-\g)}.
\end{align*}
In \cite{LMtrial,LMuniv} we   showed that  some of the phenomena
observed by Vogel and Deligne     persist in all degrees.
For example, let  $\oa$ denote the highest root of $\fg$ (here we have fixed a Cartan subalgebra 
and a set of positive roots). 
Let  $Y_k$ be the $k$-th Cartan power $\fg^{(k)}$ of $\fg$ (the module
with highest weight $k\oa$).  

\begin{theorem}\cite{LMuniv}
Use Vogel's parameters $\a,\b,\g$ as above. The $k$-th symmetric power
of $\fg$ contains three (virtual) modules
$Y_k,Y_k',Y_k''$ with Casimir eigenvalues
$2kt-(k^2-k)\a,2kt-(k^2-k)\b,2kt-(k^2-k)\g$.
Using binomial coefficients defined by
$\binom{y+x}{y} = (1+x)\cdots (y+x)/y!$, we have:
$$\tdim\, Y_k=
\frac{t-(k-\frac 12)\a}{ t+\frac{\a}2}
\frac{\binom{-\frac{2t}{\a}-2+k }{ k}
\binom {\frac{  \b-2t }\a -1+k }{k }
\binom {\frac{  \g-2t }\a -1+k }{k } 
}
{
\binom{-\frac {\b}{\a} -1+k }{k }
\binom{-\frac {\g}{\a} -1+k }{k },
}$$
and   $\tdim\, Y_k',\, \tdim\, Y_k''$ are respectively obtained by exchanging the
role  of $\a$  with $\b$ (resp. $\g$).
\end{theorem}
The modules $Y_k',Y_k''$ are described in \cite{LMuniv}. This
dimension formula is also the Hilbert function of $X^{ad}_{G([\a,\b,\g])}$.

\subsection{Cartan-Killing classification via projective geometry}
If $X=G(k,W)\subset \BP(\La k W)$ then the variety of tangent directions
to lines through a point   $E\in X$ is $Y=Seg(\BP E^*\times \BP(W/E))\subset \BP (E^*\ot W/E)
=\BP T_EX$. Moreover one can recover $X$ from $Y$ as the image of
the rational map $\BP  (T\op \BC)\dashrightarrow \pp N$ given
by  the ideals in degree $r+1$ of the varieties $\s_r(Y)$, multiplied
by a suitable power of a linear form
coming from the $\BC$-factor to give them all the same degree.
In \cite{LMsel} we showed that  the same is true for any irreducible cominuscule variety. This
enabled us to give a new, constructive proof of the classification of CHSS, without
having to first classify complex simple Lie algebras.  Moreover, a second
construction   constructs the adjoint varieties and gives a new proof
of the Killing-Cartan classification of complex simple Lie algebras without
  classifying root systems. Here is the construction for adjoint varieties:

Let $Y\subset \pp{n-2}=\BP T_1$ be a generalized cominuscule variety. Define
$Y$ to be {\it admissible} if the span of the embedded tangent lines to
$Y$, as a subvariety of the Grassmannian, has codimension one in $\La 2T_1$.
For generalized cominuscule varieties,
 this condition is equivalent to $Y$ being embedded as a {\it Legendrian variety}.
In particular $\t(Y)$ is a quartic hypersurface (for the exceptional
series, it is the quartic hypersurface described in \S\ref{sphere}).
Linearly embedded $T_1\subset \BC^{n }\subset \BC^{n+1}$ respectively
as the hyperplanes $\{x_n=0\}$ and $\{x_0=0\}$ and consider the rational
map
\begin{align*}
\phi :\pp n&\dashrightarrow \pp N\subset\BP(S^4\BC^{n+1*})\\
[x_0\hd x_n]&\mapsto
[x_0^4,x_0^3T_1^*,x_0^3x_n,x_0^2I_2(Y,\BP T_1),
x_0^2x_nT_1^*-x_0I_3(\t(Y)_{sing},\BP T_1),x_0^2x_n^2-I_4(\t(Y),\BP T_1)]
\end{align*}
In \cite{LMsel} we showed that the image is an adjoint variety and
that all adjoint varieties arise in this way.
Here are the Legendrian varieties $Y$ and the Lie
algebras of the $X^{ad}_G$ that they produce:

$$\begin{array}{rclcc}
Y & \subset & \PP^{n-2}& &\fg\\
\hline \\
v_3(\pp 1)& \subset & \pp 3 & & \fg_2\\
\pp 1\times \QQ^{m-4} & \subset & \pp {2m-5}& & \fso_m\\ 
G_{\omega}(3,6) & \subset & \pp
{13}& & \ff_4\\ G(3,6)& \subset & \pp {19}& & \fe_6\\
\BS_6& \subset & \pp {31}& & \fe_7\\
G_w(\BO^3, \BO^6)& \subset & \pp {55}& & \fe_8 .\end{array}$$
The two exceptional (i.e., non-fundamental) cases are $$\begin{array}{rclcc}
\pp {k-3}\sqcup\pp {k-3} & \subset & \pp{2k-3} & & \fsl_k\\ \emptyset &
\subset & \pp {2m-1} & & \fsp_{2m} . \end{array} $$
See \cite{LMsel,LMpop} for details. The varieties
$Y\subset \pp{n-1}$ are the asymptotic directions $B(II^{X^{ad}_G})\subset\BP T_x{X^{ad}_G}$ defined in
the next section.
 
\section{Projective differential geometry and results}\label{projgeomsect}

\subsection{The Gauss map and the projective second fundamental form}\label{IIsect}
Let $X^n\subset \BP V$ be an $n$-dimensional subvariety or complex manifold.
The {\it Gauss map} is defined by
\begin{align*}
\gamma: X &\dashrightarrow G(n+1,V)\\
x&\mapsto \hat T_xX
\end{align*}
Here  $\hat T_xX\subset V$ is the affine tangent space to $X$ at $x$,
it is related to the intrinsic tangent space $T_xX\subset T_x\BP V$ by
$T_xX=(\hat T_xX/\hat x)\ot \hat x^*\subset V/\hat x\ot \hat x^*$ where $\hat x\subset V$ is
the line corresponding to $x\in \BP V$. 
Similarly $N_xX=T_x(\BP  V)/T_xX=\hat x^*\ot (V/\hat T_xX)$. The dashed arrow is
used because the Gauss map is not defined at singular points of
$X$,  but   does define a rational map.

Now let $x\in X_{smooth}$ and consider
$$
d\gamma_x: T_xX\ra T_{\hat T_xX}G(n+1,V)\simeq (\hat T_xX)^*\ot (V/\hat T_xX)
$$
Since,  for all $v\in T_xX$, $\hat x\subset  \tker  d\gamma_x(v)$,
where $d\gamma_x(v): \hat T_xX\ra V/\hat T_xX$,  we may quotient
by $\hat x$ to obtain
$$d\underline\gamma_x\in T^*_xX\ot (\hat T_xX/\hat x)^*\ot V/(\hat T_xX)=(T^*_xX)^{\ot 2}\ot N_xX.$$
In fact, essentially because mixed partial derivatives commute, we have
$$
d\underline\gamma_x\in S^2T^*_xX\ot N_xX
$$
and we write $II_x=d\underline\gamma_x$, the {\it projective second fundamental form}
of $X$ at $x$. $II_x$ describes how $X$ is moving away from its embedded tangent
space to first order at $x$.

One piece of geometric information that $II_x$ encodes is the following: Think 
of $\BP T_xX\subset \BP T_x(\BP V)$ as the set of tangent directions in $T_x\BP V$ where
there exists a line having contact to $X$ at $x$ to order at least one. Then
$B(II_x):=\BP\{ v\in T_xX\mid II(v,v)=0\}$, often
called the set of {\it asymptotic directions}, is  the set of tangent directions
where there
exists a line having contact to $X$ at $x$ to order at least two. To study the
(macroscopic) geometry of $X$, we may study the smaller variety $B(II_x)$
and ask:
{\it What does $B(II_x)$ tell us about the geometry of $X$?} Note that $B(II_x)$ is
usually the zero set of $\tcodim X$ quadratic polynomials and thus
we expect it to have codimension equal to $\tcodim(X,\BP V)$ (assuming
the codimension of $X$ is sufficiently small, otherwise we expect it
to be empty).

\medskip

Now let $X=G/P\subset \BP V$ be a homogeneous variety. In particular
we have $II_{X,x}=II_{X,y}$ for all $x,y\in X$ so we will simply write
$II^X=II_{X,x}$. To what extent is $X$ characterized by $II^X$?

\medskip

\noindent{\bf{Aside}}.  If the ideal of a projective variety $X\subset \BP V$  is generated in degrees at most
$d$, then any line having contact with $X$ to order $d$ at a point must be contained in $X$.
By an unpublished theorem of Kostant, the ideals of
rational homogeneous varieties are generated in degree two, so $B(II^{G/P})$ corresponds
to the tangent directions to lines through a point.
Thus,  for example, when $X=G(k,V)$, $B(II_E)=Seg(\BP E^*\times \BP(V/E))$.

\subsection{Second order rigidity}
For   the Segre variety, $B(II^{Seg(\pp 2\times \pp 2)})\subset \pp 3$ is
the union of two disjoint lines ($\pp 1$'s). The Segre has codimension four,
and normally the common zero set of four quadratic polynomials on  $\pp 3$ is empty. This prompted
Griffiths and Harris to conjecture:

\begin{conjecture}\cite{GH} [Griffiths-Harris, 1979]
Let $X=Seg(\pp 2\times \pp 2)\subset \BP (\BC^3\ot \BC^3)$. 
Let $Z^4\subset \BP V$ be a variety such that at $z\in Z_{general}$, 
$II_{Z,z}=II^X$, then $Z$ is projectively equivalent to the Segre.
\end{conjecture}

\begin{theorem}\cite{Lrigid, Lchss}\label{rk2rigid} The conjecture is true, moreover the
 same result holds when $X$ is any rank two   cominuscule variety except for
$Q^n\subset \pp{n+1}$ and $Seg(\pp 1\times \pp m)\subset \BP (\BC^2\ot\BC^{m+1})$.
\end{theorem}

One can pose more generally the question: {\it Given a homogeneous variety
$G/P\subset \BP V$,   an unknown variety $Z\subset \BP W$
and a general point $z\in Z$, how many derivatives must we take
at $z$ to conclude $Z\simeq G/P$?}

\subsection{History of projective rigidity questions}
The problem of projective rigidity dates back 200 years when Monge showed $v_2(\pp 1)\subset \pp 2$, the conic curve
in the plane, is characterized by a fifth order ODE, i.e.,  it  is
rigid at order five. More recently, about 100 years ago \cite{fub},
Fubini showed that in dimensions greater than one, quadric hypersurfaces
are rigid at order three, i.e., characterized by a third order system of
PDE.

A vast generalization of Theorem \ref{rk2rigid} was obtained by
Hwang and Yamaguchi:

\begin{theorem}\cite{HY}\label{hythm} Let $X\subset \BP V$ be an irreducible
homogeneously embedded CHSS,  other than a quadric hypersurface or
projective space,  with osculating sequence of length $f$. Then
$X$ is rigid at order $f$.
\end{theorem}

See \S\ref{oscfiltsect} for the definition of the osculating sequence.

Even more exciting than the theorem of Hwang and Yamaguchi are
the methods they used to prove it. More on this in \S\ref{betterway}.

If one changes the hypotheses slightly, one gets an order two result:
\begin{corollary}\cite{Lima}\label{hycor} Let $X\subset \BP V$ be a
cominuscule variety,  other than a quadric hypersurface. Let $Y\subset \BP W$ be an unknown variety such that
$\tdim \langle Y\rangle=\tdim   V$, and such that for
$y\in Y_{general}$, $II_{Y,y}=II^X$.
Then $Y$ is projectively equivalent to $X$.
\end{corollary}

The proof of this result uses two facts: that the higher fundamental
forms of cominuscule varieties are the (full) prolongations of the
second, and that any variety with such fundamental forms must
be the homogeneous model (which follows from Theorem \ref{hythm}). See \cite{Lima} for details.

\subsection{Rigidity and flexibility of adjoint varieties}
For the adjoint varieties, it is easy to see that order two rigidity fails
(see \cite{LM0}), even though they have osculating
sequence of length two. These lectures will be centered around the proof of the following
theorem:

\begin{theorem}\cite{LRrigid}\label{adjrigidthm} For simple groups $G$,
 the adjoint varieties $X^{ad}_{G}\subset \BP\fg$ (other than   $G=A_1$)
are rigid at order three.
\end{theorem}
In the case $G=A_1$, $X^{ad}_{A_1}=v_2(\pp 1)$, which Monge showed to be rigid at
order five but not   four.

Robles and I originally wrote a \lq\lq brute force\rq\rq\ proof of this theorem
in December 2006,
although we had been attempting to use the methods of Hwang and Yamaguchi. Finally,
when A. Cap visited us in June 2007, in what can only be described as an incredible
syncronicity, we made the breakthrough needed, in parallel with Cap making a breakthrough
in his work on BGG operators with maximal kernel. In \S\ref{Liealgcohsect} I   describe the methods,
which involve a reduction to a Lie algebra cohomology calculation,  and
which should be useful
for other EDS questions.  I conclude this section with
the description of a result that was obtained using traditional
EDS techniques:

The adjoint varieties are the homogeneous models for certain {\it parabolic geometries},
a much discussed topic at this conference. In particular
they are equipped with an intrinsic geometry that includes a holomorphic contact structure.
All the intrinsic geometry is visible at order two (including the distinguished hyperplane)
{\it except} for the contact structure. This inspires the modified question:

Assume $Z\subset\BP V$ is such that at $z\in Z_{general}$ we have
$II_{Z,z}=II^{X^{ad}_G}$ and the resulting hyperplane distribution is contact,
can we conclude $Z\simeq X^{ad}_G$? Of course for $G=A_1$, we know the 
answer is no, thanks to Monge.

\begin{theorem}\cite{LRrigid} If $G\neq A_1,A_2$, then YES! If
 $G=SL_3=A_2$, then NO!; there exist \lq\lq functions worth\rq\rq\ of
impostors.
\end{theorem}

\begin{remark} Although the results are formulated in the  holomorphic
category, the exact same result holds in the real analytic category.
\end{remark}

The second conclusion has interesting consequences for geometry. A $3$-manifold
$M$ equipped with a contact distribution which has two distinguished line
sub-bundles is the path space for a {\it path geometry} in the plane. Such structures
have two \lq\lq curvature \rq\rq\ functions, call them $J_1,J_2$, which are
differential invariants that measure the difference between $M$ and the
homogeneous model, which is
$X^{ad}_{SL_3}=Flag_{1,2}(\BC^3)$. This geometry has been well studied by many
authors, including E. Cartan \cite{cartan}. For example, if $J_1\equiv 0$, then the paths are the geodesics
of a projective connection. See \cite{IvL}, Chapter 8 for more.

\begin{theorem}\cite{LRrigid} The general impostor above has $J_1,J_2$ nonzero,
 although they do satisfy   differential relations.
\end{theorem}

This is interesting because it is difficult to come up with natural restrictions
on the invariants $J_1,J_2$ short of imposing that one or the other is zero.
 A classical analog of this situation, where the condition of being extrinsically
realizable gives rise to a natural
system of PDE,  is the set of surfaces equipped with Riemannian metrics
that admit a local isometric immersion into Euclidean $3$-space such that the
image is a minimal surface. Ricci discovered that for this to happen, the 
Gauss curvature $K$ of the 
Riemannian metric of the surface must satisfy
the PDE
$$
\Delta log(-K)=4K
$$
where $\Delta$ is the Riemannian Laplacian.
See \cite{CO,IvL} for details.

\section{From PDE to EDS}\label{pdeedssect}
Exercise: show that any system of PDE can be expressed as a first order
system. (Hint: add variables.) Thus we only discuss first order systems.
We want to study them from a   geometric  perspective, that of
submanifold geometry.

Let $\BR^n$ have coordinates $(x^1\hd x^n)=(x^i)$ and
$\BR^m$ coordinates $(u^1\hd u^m)=(u^a)$, let
\be
\label{pdesys}
F^r(x^i,u^a,p^a_i)=0 \ \ 1\leq r\leq R
\ene
be a system of equations in $n+m+nm$ unknowns.
We view this as a system of PDE by stating that
a map
\begin{align*}
 f:\BR^n&\ra \BR^m\\ x&\mapsto u=f(x)
\end{align*}
is a solution of the system determined by \eqref{pdesys}
if \eqref{pdesys} holds when we set $u^a=f^a(x)$ and  $p^a_i=\frp{u^a}{x^i}$.

To rephrase slightly, let $J^1(\BR^n,\BR^m):=\BR^n\times \BR^m\times \BR^{nm}$
have coordinates $(x^i,u^a,p^a_i)$. Consider the differential forms
$$
\th^a:=du^a-p^a_idx^i\in\Omega^1(J^1(\BR^n,\BR^m)) \ \ 1\leq a\leq m
$$
Then we have the following correspondences:

\begin{center} 
\begin{tabular}{|c|c|c|}
Graphs of maps $f:\BR^n\ra \BR^m$,
&\ $\leftrightarrow$\ 
& 
immersions $i:M^n\ra J^1(\BR^n,\BR^m)$ such that
\\ 
$\Gamma_f\subset \BR^n\times \BR^m$ 
& & 
$i^*(\th^a)=0$ and $i^*(dx^1\ww\cdots\ww dx^n)$ is nonvanishing.
\end{tabular}

\bigskip

\begin{tabular}{|c|c|c|} 
Graphs of maps $f:\BR^n\ra \BR^m$,
&
&
immersions $i:M^n\ra \Sigma\subset  J^1(\BR^n,\BR^m)$ 
\\
$\Gamma_f\subset \BR^n\times \BR^m$ satisfying
&\ $\leftrightarrow$\  &
such that
$i^*(\th^a)=0$ and $i^*(dx^1\ww\cdots\ww dx^n)$  is
\\
 the PDE system determined by\eqref{pdesys} 
& &
 nonvanishing, where $\Sigma$ is the zero set of \eqref{pdesys}.
\end{tabular}
\end{center}

\medskip

Now we are ready for EDS:

\medskip

\begin{definition} A {\it Pfaffian EDS with independence condition} on 
a manifold $\Sigma$ is a sequence 
 of sub-bundles   $I\subset J\subseteq T^*\Sigma$. Write
$n=\trank(J/I)$.

An {\it integral manifold} of $(I,J)$ is an immersed $n$-dimensional submanifold $i: M\ra \Sigma$
such that $i^*(I)=0$ and $i^*(J/I)=T^*M$.
\end{definition}

In the motivating example we had $I=\{\th^a\}$ and $J=\{\th^a,dx^i\}$.

Thus we have transformed questions about the existence of solutions to a system
of PDE to questions about the existence of submanifolds tangent to a distribution.

We next show how to determine existence. But first, here are a few
successes of EDS:

\begin{itemize}

\item  Determination of existence of local isometric embeddings of (analytic) Riemannian manifolds
into Euclidean space and other space forms. (e.g., Cartan-Janet theorem),
see, e.g.,  \cite{BBG,BCG3}

\item Proving the existence of 
Riemannian manifolds with holonomy $G_2$ and $Spin_7$  (Bryant \cite{Bryanthol}).

\item  Rigidity/flexibility of Shubert varieties in Grassmannians and other 
symmetric spaces (Bryant \cite{BryantSchubert}).

\item Proving existence of special Lagrangian and other calibrated submanifolds
(Harvey and Lawson, \cite{HL}).

\end{itemize}

\section{The Cartan algorithm to determine local existence of integral manifolds
to an EDS}\label{cartanalgsect}

The essence of the Cartan algorithm is to systematically understand the additional
conditions imposed by a system of PDE by the fact that mixed partial derivatives
commute. In the language of differential forms, this is the statement
$$
i^*(\th)=0\ \Rrightarrow \ i^*(d\th)=0\ \forall\th\in I.
$$
For example, in \S\ref{pdeedssect},  $i^*(d\th)=0$ forces
$\partial p^a_{i}/\partial x^j= \partial p^a_{j}/\partial x^i$ and
on integral manifolds $p^a_i=\partial u^a/\partial x^i$.

\subsection{Linear Pfaffian systems}Among Pfaffian systems, there are those where the set of integral elements through a point
forms an affine space,   the {\it linear} systems.
\begin{definition}
 A Pfaffian EDS is linear if the map
\begin{align*}
 I&\ra \La 2(T^*\Sigma/J)\\
\th&\mapsto d\th \tmod J
\end{align*}
is zero.
\end{definition}

To simplify the exposition we will restrict to linear Pfaffian systems. (This is theoretically
no loss of generality, see \cite{IvL}, Chapter 5.)

Although some of the theory is valid in the
$C^{\infty}$ category (see, e.g., \cite{yang}), we will   work in the real or complex analytic category where the theory
works best, and in the applications of this paper, we will actually work in the holomorphic category.
In particular, it makes sense to talk of a {\it general point} of an analytic manifold,
where general is with respect to the EDS on it (e.g., points where the system does
not drop rank, where the derivatives of the forms in the system don't drop rank, etc.)

Fix $x\in \Sigma_{general}$.  To determine
the integral manifolds through $x$ we work infinitesimally and reduce to problems
in linear algebra  (as one does with most problems in mathematics). 

\begin{definition} An $n$-plane $E\subset T_x\Sigma$ is called an {\rm integral
 element} if $\th_x\mid_E=0$ and $d\th_x\mid_E=0$ for all $\th\in I$.
\end{definition}

Let $\cV(I)_x\subset G(n,T_x\Sigma)$ denote the set of all integral elements at $x$. 
As remarked above, if $(I,J)$ is linear, then $\cV(I)_x$ is an affine space.
Set
\begin{align*}
 V&=(J/I)_x^*\\
W&= I_x^*.
\end{align*}
Fix a splitting $T^*_x\Sigma=J_x\op J^c_x$ and define a bundle map
\begin{align*}
W^*&\ra \La 2 V\\
\th_x&\mapsto d\th_x\tmod I, J^c
\end{align*}
we may consider this map as a tensor $T\in W\ot \La 2 V^*$, which we call
the {\it apparent torsion} of $(I,J)$ at $x$.
Since the apparent torsion changes if we change the splitting, we instead
consider 
\be\label{toreqn}
[T]\in W\ot \La 2V^*/\sim
\ene
called the {\it torsion} of the system at $x$, which is well defined.
The equivalence $\sim$ is precisely over the different choices of splittings
and is   made explicit in \eqref{toreqne} below.

Since on the one hand we are requiring $I$ to vanish on  integral elements
but $J/I$ to be of maximal rank, if $[T]\neq 0$,   there
are no integral elements over $x$, i.e., $\cV(I)_x=\emptyset$.
If this is the case, we start over on the submanifold (analytic subvariety)
$\Sigma'\subset \Sigma$ defined by $[T]=0$.

Now consider the bundle map given by exterior differentiation,
$\th\mapsto d\th \tmod I$, a component of which is
$I\ra (T^*\Sigma/J)\ot J$.   Pointwise this is a linear map
$W^*\ra (T^*\Sigma/J)_x\ot V^*$, which we may consider
as a linear map
$(T^*\Sigma/J)_x^*\ra W\ot V^*$.
Let
$$
A\subset W\ot V^*
$$
denote the image of this map at $x$, which is called
the {\it tableau} of $(I,J)$ at $x$.
For linear Pfaffian systems $A$ corresponds to $\cV(I)_x$, where
we transform the affine space to a linear space by picking a 
base integral element $\th^a_x=\pi^{\ep}_x=0$, where the $\th^a_x$ give a basis of
$I_x$ and $\pi^{\ep}_x$ give
a basis for a choice of $J^c_x$. The quantity $\tdim A$ gives us   an  answer
to the infinitesimal version of the question: {\it How many integral manifolds of $(I,J)$ pass through $x$?}

\begin{example} Say $J=T^*\Sigma$, which is the situation of the Frobenius theorem.
Then there exist integral manifolds iff $T=[T]=0$. Here $A=(0)$ always (which
corresponds to the uniqueness part of the theorem).
\end{example}

We may think of the tableau $A$ as parametrizing the
choices of admissible first order terms in
the Taylor series of an integral manifold at $x$ expressed in terms of a graph.
From this perspective, the next question is: What are the admissible second order
terms in the Taylor series?
At the risk of being repetitive, the condition to check is that:
{\it Mixed partials commute!}

\subsection{Prolongations and the Cartan-K\"ahler theorem}
Let 
$$
\d: W\ot V^*\ot V^*\ra W\ot \La 2 V^*
$$
denote the skew-symmetrization map.
Define 
$$
A\up 1:=\tker\d\mid_{A\ot V^*} = (A\ot V^*)\cap (W\ot S^2V^*)
$$
the {\it prolongation} of $A$.
We may think of $A\up 1$ as parametrizing the admissible second order terms
in the Taylor series. 

At this point we can make explicit the equivalence in \eqref{toreqn}. It is
\be\label{toreqne}
[T]\in W\ot \La 2V^*/\d(A\ot V^*)
\ene 

Now we know how to determine the admissible third order Taylor
terms etc..., but should we keep going on forever? When can we stop working?
The answer is given by the following theorem:

\begin{theorem}[Cartan, Cartan-K\"ahler](see, e.g., \cite{BCG3,IvL})
Let $(I,J)$ be an analytic linear Pfaffian system on $\Sigma$, let $x\in \Sigma_{general}$.
Assume $[T]_x=0$.
Choose an $A$-generic flag $V^*=V^0\supset V^1\supset \cdots \supset V^{n-1}\supset 0$.
Let $A_j:=A\cap (W\ot V^j)$. Then
$$
\tdim A\up 1\leq \tdim A+\tdim A_1+\cdots + \tdim A_{n-1}.
$$
If equality holds then we say $(I,J)$ is {\rm involutive} at $x$ and then
there exist local integral manifolds through $x$ that depend roughly
on $\tdim (A_r/A_{r-1})$ functions of $r$ variables, where
$r$ is the unique integer such that $A_{r-1}\neq A_r=A_{r+1}$.
\end{theorem}

\bigskip
\subsection{Flowchart and exercises} Here $\Omega\in \La n(J/I)$ encodes the
independence condition:

{\small
\setlength{\unitlength}{.8mm}
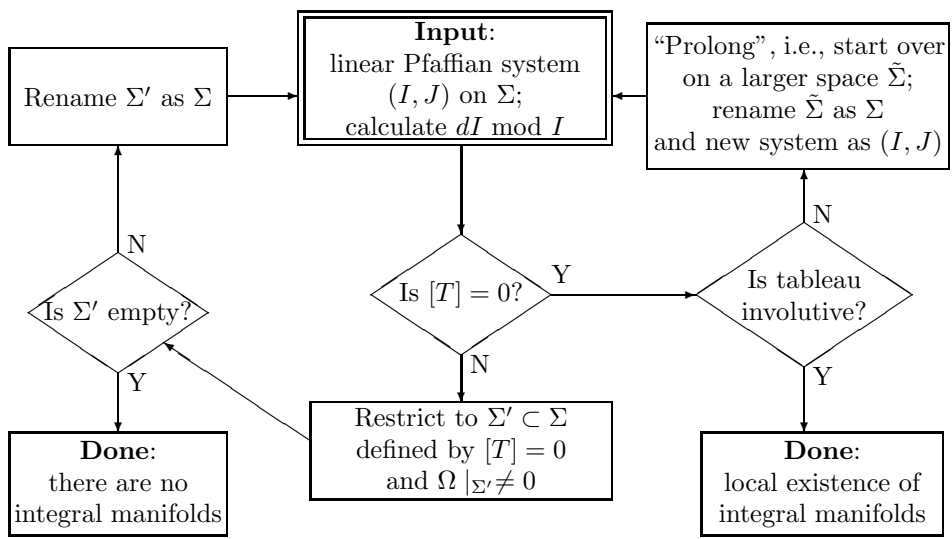
\begin{center}
\begin{picture}(156,90)(0,-5)
\put(0,60){\framebox(36,16){Rename $\Sigma'$ as $\Sigma$}}
\put(36,68){\vector(1,0){12}}
\put(48,60){\framebox(52,22){\framebox(50,20){\Cbox{4cm}{
{\bf Input}:\\ linear Pfaffian system\\
 $(I,J)$ on $\Sigma$;\\
calculate $dI \tmod I$
}}}}

\put(106,68){\vector(-1,0){6}}
\put(106,56){\framebox(50,24){\Cbox{4cm}{
``Prolong'', i.e., start over\\ on a larger space $\tilde\Sigma$; \\
rename  $\tilde\Sigma$ as $\Sigma$\\ and new system as  $(I,J)$
}}}

\put(18,42){\vector(0,1){18}\ N}
\multiput(18,42)(15,-10){2}{\line(-3,-2){15}}
\multiput(18,42)(-15,-10){2}{\line(3,-2){15}}
\put(5,27){\makebox(26,10){Is $\Sigma'$ empty?}}

\put(75,60){\vector(0,-1){15}}
\multiput(75,45)(15,-10){2}{\line(-3,-2){15}}
\multiput(75,45)(-15,-10){2}{\line(3,-2){15}}
\put(60,30){\makebox(30,10){Is $[T]=0$?}}
\put(90,35){\vector(1,0){24}}\put(90,37){Y}

\put(132,47){\vector(0,1){9}\ N}
\multiput(132,47)(18,-12){2}{\line(-3,-2){18}}
\multiput(132,47)(-18,-12){2}{\line(3,-2){18}}
\put(117,30){\makebox(30,10){\Cbox{3cm}{Is tableau involutive?}}}

\put(18,22){\vector(0,-1){10}}\put(18,19){\ Y}
\put(0,-5){\framebox(36,17){\Cbox{2.8cm}{{\bf Done}:\\ there are no integral manifolds}}}

\put(75,25){\vector(0,-1){8}}\put(75,22){\ N}
\put(50,1){\framebox(50,16){\Cbox{4cm}{
Restrict to $\Sigma'\subset \Sigma$\\
defined by $[T]=0$\\ and $\Omega\mid_{\Sigma'}\neq 0$
}}}
\put(50,10.666){\vector(-3,2){24.5}}

\put(132,23){\vector(0,-1){11}}\put(132,20){\ Y}
\put(115,-5){\framebox(40,17){\Cbox{3.2cm}{{\bf Done}:\\ local existence of integral manifolds}}}
\end{picture}
\end{center}}

\exerset{Set up the EDS  and perform the Cartan algorithm in the following problems:
\item  The Cauchy-Riemann equations $u_x=v_y, u_y=-v_x$. (Work on a codimension
two submanifold of $J^1(\BR^2,\BR^2)$.)
\item Find all surfaces $M^2\subset \BE^3$ such that every point is an umbillic point.
\item Determine the local existence of special Lagrangian submanifolds of $\BR^{2n}\simeq \BC^n$.
\item (For the more ambitious.) Pick your favorite $G\subset SO(p,q)$ and determine
local existence of pseudo-Riemannian manifolds with holonomy $\subseteq G$.\
\item After you read \S\ref{projframessect}, show that $Seg(\pp 2\times \pp 2)\subset \pp 8$ is
rigid to order two. Then roll up your sleeves to show that $Seg(\pp 1\times \pp n)$ is
flexible at order two.
}

\subsection{For fans of bases}
Here is a recap in bases: take a local coframing of $\Sigma$
adapted to the flag $I\subset J\subset T^*\Sigma$, i.e., write
$I=\{\th^a\}$, $1\leq a\leq \trank I$, $J=\{\th^a,\o^i\}$, $1\leq i\leq \trank(J/I)$,
$T^*\Sigma=\{\th^a,\o^i,\pi^{\ep}\}$, $1\leq \ep\leq \trank(T^*\Sigma/J)$.
Then there exist functions $A\hd H$ such that
\begin{align*}
d\th^a=& A^a_{\ep i}\pi^{\ep}\ww\o^i + T^a_{ij}\o^i\ww\o^j + E^a_{\ep,\d}\pi^{\ep}\ww\pi^{\d}\\
&+F^a_{bi}\th^b\ww\o^i + G^a_{b\ep}\th^b\ww\pi^{\ep} +H^a_{bc}\th^b\ww\th^c.
\end{align*}
Since we only care about $d\th^a\tmod I$ we ignore the second row.
The system is linear iff $E^a_{\e\d}=0$.
The   apparent torsion is
$T=T^a_{ij}v^i\ww v^j\ot w_a\in \La 2 V^*\ot W$.
The tableau is
$$A=\{A^a_{\ep i}v^i\ot w_a\mid 1\leq \ep\leq \trank(T^*\Sigma/J)\}\subset V^*\ot W
$$
The torsion is
\begin{align*}
[T]=&T^a_{ij}w_a\ot v^i\ww v^j\tmod \{(A^a_{j\ep}e^{\ep}_i-A^a_{i\ep}e^{\ep}_j)w_a\ot v^i\ww v^j
\mid  e^{\ep}_j\in \BF \}\\
&\in W\ot \La 2 V^*/\d(A\ot V^*).
\end{align*}
Here $\BF= \BR$ or $\BC$.

\section{Moving frames for submanifolds of projective space}\label{projframessect}

Let $U=\BC^{N+1}$. Let $G\subseteq GL(U)$ have
{\it Maurer-Cartan form} $\o_{\fg}\in\Omega^1(G,\fg)$. Recall that 
the Maurer-Cartan form has the following properties:
\begin{itemize}
\item left invariance: $L_g^*\o=\o$ where $L_g: G\ra G$ is the map $a\mapsto ga$.
\item $\o_{Id}: T_{Id}G\ra \fg$ is the identity map
\item $d\o=-\o\ww \o$ or equivalently, $d\o=-\frac 12 [\o,\o]$ (Maurer-Cartan equation) 
\end{itemize}

(Here $[\o,\eta](v,w):=[\o(v),\eta(w)]-[\o(w),\eta(v)]$.)

\subsection{Adapted frame bundles}
We want to study the geometry of submanifolds $Y\subset \BP U$ from the perspective
of Klein, that is we consider  $Y\sim Z$ if there exists $g\in GL(U)$ such that
$g.Y=Z$. In order to efficiently incorporate the group action, we will
work \lq\lq upstairs\rq\rq\ on $GL(U)$. Consider the projection map
\begin{align*}
\pi: GL(U)&\ra \BP U\\
(e_0\hd e_N)&\mapsto [e_0]
\end{align*}
where we view the $e_j$ as column vectors. Fixing a reference basis,
we may identify $GL(U)$ with the set of all bases of $U$. We will restrict ourselves
to   submanifolds of $GL(U)$ consisting of bases adapted to the local
differential geometry of $Y\subset \BP U$.
First, consider $\cf^0_Y:=\pi\inv(Y)$, the $0$-th order adapted frames (bases).
Let $n=\tdim Y$.
Next consider
$$
\cf^1_Y:=\{ f=(e_0\hd e_N)\in \cf^0_Y \mid \hat T_{[e_0]}Y=\{ e_0\hd e_n\}\}
$$ 
the frames adapted to the flag $\hat x\subset \hat T_xY\subset U$ over
each point, called the {\it first order adapted frame bundle}.
Write $L=\hat x$, $T=\hat T_xY/\hat x$, $N=U/\hat T_xY$.  
Adopt index ranges $1\leq \a,\b\leq n=\tdim T_xY$, $n+1\leq \mu,\nu\leq \tdim U-1$.
Write $\fgl(U)= (L \op T\op N)^*\ot (L\op T\op N)$
and let, for example,  $\o_{L^*\ot T}$ denote the component
of $\o$ taking values in $L^*\ot T\subset U^*\ot U=\fgl(U)$.
Write
\be
\o_{\fgl(U)}=\begin{pmatrix}
\ooo 00&\ooo 0\b&\ooo 0 \nu\\
\ooo \a 0&\ooo \a\b&\ooo\a\nu\\
\ooo \mu 0&\ooo\mu\b&\ooo\mu\nu
\end{pmatrix} =\begin{pmatrix}
\o_{L^*\ot L}&\o_{T^*\ot L}&\o_{N^*\ot L}\\
\o_{L^*\ot T}&\o_{T^*\ot T}&\o_{N^*\ot T}\\
\o_{L^*\ot N}&\o_{T^*\ot N}&\o_{N^*\ot N}
\end{pmatrix} \, .
\ene
Write $i: \cf^1_Y\ra GL(U)$ as the inclusion.
We have $i^*(\ooo\mu 0)=i^*(\o_{L^*\ot N})=0$.
(Note that at each $f\in \cf^1_Y$ we actually have a splitting
$U=L\op T\op N$.)

\subsection{Fubini Forms}\label{fubsect}
Now anytime you ever see a quantity equal to zero, {\it Differentiate it!}
We have
$$
i^*(\oo\mu)=0 \Rrightarrow i^*(d\oo\mu)=0
$$
which, using the Maurer-Cartan equation tells us
that
$$
i^*(\ooo\mu\a\ww\oo\a)=0
$$
(note use of summation convention).
We are assuming that the forms $i^*(\oo \a)$ are linearly independent
(as they span the pullback of $T^*Y$ by our choice of adaptation) so
we must have
$$
i^*(\ooo\mu\a)=\qq\mu\a\b i^*(\oo\b)
$$
for some functions $\qq\mu\a\b:\cf^1_Y\ra \BC$.
Moreover (exercise) $\qq\mu\a\b=\qq\mu\b\a$ for all $\a,\b$ (this is often
called the {\it Cartan Lemma}).
The functions $\qq\mu\a\b$ vary on the fiber, but they do contain geometric
information. If we form the tensor field
$$
F_2:=\qq\mu\a\b\oo\a\circ\oo\b\ot e^0\ot( e_{\mu} \tmod \hat T_xY)
\in \Gamma(\cf^1_Y, \pi^*(S^2T^*Y\ot NY))
$$
a short calculation shows that $F_2$ is constant on the fibers, i.e.
$F_2=\pi^*(II)$ for some tensor $II\in \Gamma(Y,S^2T^*Y\ot NY)$.
$II$ is indeed the projective second fundamental form defined
as the derivative of the Gauss map in \S\ref{projgeomsect}.

Unlike with the case of the Gauss map, where it was not clear how to continue, here
it is - we have a quantity equal to zero: $\ooo\mu\a-\qq\mu\a\b\oo\b$ so we differentiate
it! (From now on we drop the $i^*$ when describing pullbacks
of differential forms to simplify notation.) The result is that there exist functions
$$
\rr\mu\a\b\g: \cf^1_Y\ra \BC
$$
such that 
$$
d\qq\mu\a\b-\qq\mu\a\b\ooo 00-\qq\nu\a\b\ooo\mu\nu+\qq\mu\a\ep\ooo\ep\b+\qq\mu\b\ep\ooo\ep\a
=\rr\mu\a\b\g\oo\g
$$
which gives rise to a tensor field
$$
F_3\in \Gamma(\cf^1_Y,\pi^*(S^3T^*Y\ot NY))
$$
This tensor, called the {\it Fubini cubic form} does not descend to be well defined on $Y$,  but it does contain
important geometric information.

\subsection{Second order Fubini systems}\label{sec:Fub2}
Fix vector spaces $L,T,N$ of dimensions $1,n,a$ and
fix an element $F_2\in S^2T^*\ot N\ot L$. Let $U = L\op T\op N$, and let $\o\in \O^1(GL(U),\fgl(U))$ denote the Maurer-Cartan form. 

Writing the Maurer-Cartan equation
component-wise yields, for example,
$$
d\o_{L^*\ot T}=-
\o_{L^*\ot T}\ww\o_{L^*\ot L}
-\o_{T^*\ot T}\ww\o_{L^*\ot T}
-\o_{N^*\ot T}\ww\o_{L^*\ot N}.
$$
 
Given $F_2 \in L \ot S^2T^* \ot N$, 
the {\it second order Fubini system} for $F_2$ is
$$I_{\tFub_2}=\{\o_{L^*\ot N}, \o_{T^*\ot N}-F_2(\o_{L^*\ot T})\},\ \ 
 J_{\tFub_2}=\{I_{\tFub_2}, \o_{L^*\ot T}\}.
$$
Its integral manifolds are submanifolds $\cF^2 \subset GL(U)$ that
are adapted frame bundles of submanifolds $X\subset \BP U$
having the property that at each point $x\in X$, the projective
second fundamental form $F_{2,X,x}$ is equivalent to $F_2$.
(The tautological system for frame bundles of
arbitrary $n$ dimensional submanifolds
is given by $I=\{\o_{L^*\ot N}   \}$, 
$ J=\{I, \o_{L^*\ot T}\}$.)

Let $R\subset GL(L)\times GL(T)\times GL(N)$
denote the subgroup stabilizing $F_2$ and
let 
$$\fr\subset (L^*\ot L)\op (T^*\ot T)\op 
(N^*\ot N) =:\fgl(U)_{0,*}
$$
denote its subalgebra.  These are the elements of $\fgl(U)_{0,*}$ annihilating $F_2$. (The motivation for the notation $\fgl(U)_{0,*}$ is explained in \S\ref{oscfiltsect}.)  Assume $\fr$ is reductive
so that we may decompose $\fgl(U)_{0,\ast}=\fr\op\fr\upperp$ as an $\fr$-module.

\medskip

In the case of homogeneous varieties $G/P$, 
$F_2\in S^2T^*\ot N\ot L$ will correspond to a trivial representation
of the Levi factor of $P$, which we denote $G_0$. For example, let $G/P=G(2,M)\subset \BP \La 2 M$
be the Grassmannian of $2$-planes. Then $R=G_0= GL(E)\times GL(F)$,
$T=E^*\ot F$, $N=\La 2 E^*\ot \La 2 F$, and
we have the decomposition $S^2T^*=(\La 2 E\ot \La 2 F^*)\op(S^2E\ot S^2F)$,
and $F_2\in S^2T^*\ot N$ corresponds to the trivial representation
in $ (\La 2 E\ot \La 2 F^*)\ot (\La 2 E\ot \La 2 F^*)^*$.

\medskip

In the notation of \S\ref{cartanalgsect},
\begin{displaymath}
  V \ \simeq \ L^*\ot T \, , \quad 
  W \ \simeq \ (L^*\ot N)\op (T^*\ot N) \, , \quad 
  A \ \simeq \ \fr\upperp \, , 
\end{displaymath}
with $L^*\ot N\subset W$ in the first derived system.
That  $\fr\upperp \subset V^*\ot W$    may be seen as follows
\begin{eqnarray}
  \nonumber 
  d(\o_{T^*\ot N}-F_2(\o_{L^*\ot T})) & = & 
  - \ \o_{T^*\ot L}\ww\o_{L^*\ot N} \ - \ \o_{T^*\ot T}\ww\o_{T^*\ot N} \\
  & & \nonumber
  - \ \o_{T^*\ot N}\ww\o_{N^*\ot N} \ 
  + \ F_2 ( \o_{L^*\ot L}\ww\o_{L^*\ot T} ) \\
  & & \nonumber
  + \ F_2 ( \o_{L^*\ot T}\ww\o_{T^*\ot T} \ 
  + \ \o_{L^*\ot N}\ww\o_{N^*\ot T} ) \\
  & \equiv & \label{f2calc}
  - \o_{T^*\ot T}\ww F_2(\o_{L^*\ot T}) \ 
  - \ F_2(\o_{L^*\ot T})\ww\o_{N^*\ot N} \\ \nonumber
  & & - F_2(-\o_{L^*\ot L}\ww\o_{L^*\ot T} \ 
      - \ \o_{L^*\ot T}\ww\o_{T^*\ot T}) \ \tmod I  \\ \nonumber
 & \equiv &  (\o_{0,\ast} \, . \, F_2)\ww\o_{L^*\ot T} \ \tmod I \\ \nonumber
 & \equiv & (\o_{\fr\upperp} \, . \, F_2)\ww\o_{L^*\ot T} \ \tmod I \, .
\end{eqnarray}
To understand the last two lines, $\w_{0,\ast} \, . \, F_2$ denotes the action of the $\fgl(U)_{0,\ast}$--valued component $\w_{0,\ast}$ of the Maurer-Cartan form on $F_2 \in S^2T^* \ot N$.  Recall that $\fr$ is the annihilator of this action.    By definition $\w_{0,\ast} \, . \, F_2 = (\w_\fr + \w_{\fr^\perp}) \, . \, F_2= \w_{\fr^\perp} \, . \, F_2$.

\medskip

 For the Cartan algorithm we need to calculate
$A\up 1=\tker\d$ where
$$
\d: \fr\upperp\ot V^*\ra W\ot \La 2 V^*
$$

One can check directly that $A$ is {\it never} involutive for any $F_2$ system. (One has not yet uncovered all commutation relations among mixed partials. This is
essentially because we have yet to look at the entire Maurer-Cartan form).

Thus we need to prolong, introducing elements of $A\up 1$ as new variables and
differential forms to force variables representing the elements of $A\up 1$ to behave properly, just as the 
$\th^a$'s forced the $p^a_i$'s to be derivatives in \S\ref{pdeedssect}.

Before doing so, we simplify our calculations by exploiting the group action
to normalize $A\up 1\sim F_3$ as much as  possible.
Write $\fgl(U)_{1,*}:=T^*\ot L+N^*\ot T$.
Consider the linear map
$$
\d: \fgl(U)_{1,*}\ra (L^*\ot T)\ot \fr\upperp = V^*\ot A
$$
defined as the transpose of the Lie bracket
$$
[,]: \fgl(U)_{1,*}\times L\ot T^*\ra \fr\upperp\subset \fgl(U)_{0,*}
$$
Now $L\ot T^*\subset \fgl(U)_{-1.*}:= \fgl(U)_{1.*}^*$.
Then we define
$$
A^{(1)}_{red}
:=\frac{\tker\d: A\ot V^*\ra W\ot \La 2 V^*}
{\tim \d: \fgl(U)_{1,*}\ra A\ot V^*}.
$$
One can calculate directly, that
when $X$ is a rank $2$ CHSS in its minimal homogeneous embedding
(other than a quadric or $\pp 1\times \pp m$)
and $\overline F_2=II^X$, that $A^{(1)}_{red}=0$.
In  these cases, we begin again with a new system
$$
\tilde I:=\{I,\o_{\fr\upperp}\}
$$
on $GL(U)$.
Again, one can check that $\tilde A$ is never involutive, but that
$\tilde {A}_{red}\up 1=0$.
Finally, one defines
$$
\tilde{\tilde I}=\{\tilde I, \o_{\fr\upperp}\}
$$
which turns out to be Frobenius in the case of rank 2 CHSS, i.e. $\tilde{\tilde A}\up 1=0$,
which implies rigidity.

\subsection{An easier path to rigidity?}\label{betterway}
A better way to obtain the same conclusion is to observe that
$A^{(1)}_{red}$ looks like the graded Lie algebra cohomology
group $H^1_1(\fg_{-},\fg\upperp)$ defined in \S\ref{Liealgcohsect}. In the CHSS case, it
indeed is this cohomology group, but in all other cases
it is not. In the next few sections we will see that the corrrespondence
is exact in the 
CHSS case, and how it fails in all other cases - it fails in two
ways, but none the less, with the introduction of certain
{\it filtered} EDS, the use of Lie algebra cohomology can be recovered.

\section{Osculating gradings and root gradings}\label{gradingsect}

As mentioned above, for homogeneously embedded CHSS, the osculating
filtration  and a filtration induced by the Lie algebra coincide,  but that
these two differ for all other homogeneous varieties. In this section we
explain the two filtrations.

 
\subsection{The osculating filtration}\label{oscfiltsect}

Given a submanifold $X\subset \BP U$, and $x\in X$, the {\it osculating filtration at $x$}
$$U_0\subset U_1\subset\cdots \subset U_r=U$$
is defined by 
\begin{align*}
U_0&=\hat x,\\
U_1&=\hat T_xX,\\
U_2&=U_1+F_2(L^* \ot S^2T_xX)\\
&\vdots\\
U_r&=U_{r-1}+F_r(L^*{}^{\ot (r-1)}\ot S^rT_xX).
\end{align*}
   We may reduce the frame bundle $\cF^1_X$ to framings adapted to the osculating sequence by 
restricting to  
$e = (e_0 , e_\a , e_{\m_2} , \ldots , e_{\m_f}) \in \cF^1_X$ such that $[e_0] \in X$, $\hat T_{[e_0]} X = \tspan \{ e_0 , e_\a \}$ and
$U_k = \tspan \{ e_0 , e_\a , e_{\m_2} , \ldots , e_{\m_k} \}$.  (The indices $\a$ and $\m_j$ 
respectively range over $1 , \ldots , n$ and $\tdim U_{j-1} + 1 , \ldots , \tdim U_j$.)  From now on we work on this reduced frame-bundle, denoted $\cF^r_X \subset \cF^1_X$.   

At each point of $\cF^r_X$ we obtain a splitting of $U$.  This induces a splitting
$$
\fgl(U)=\oplus\fgl(U)_{k,\ast} \, .
$$
(The asterisk above is a place holder for a second  splitting given by the representation theory when $X = G/P$
that we define in \S\ref{rootgradingsect}.) 

The osculating filtration of $U$   determines a refinement of the Fubini forms.  Let $N_k=U_k/U_{k-1}$ and define $F_{k,s}: N_k^*\ra  L^{\ot (s-1)}\ot S^sT^*_xX$
by restricting $F_s\in N_k*\ot L^{*\ot (s-1)}\ot S^sT_xX$ to $N_k^*$. 
Although the Fubini forms do not descend to well-defined tensors on $X$, the {\it fundamental forms} $F_{k,k}$ do.  By construction, $F_{k,k} : L^*{}^{\ot (k-1)} \ot S^kT_xX \to N_{k,x}X$ is surjective.     
\medskip

\subsection{The root grading}\label{rootgradingsect}

Let $\tilde\fg$ be a complex semi-simple Lie algebra with a fixed set of simple roots $\{ \a_1\hd \a_r \}$, and corresponding fundamental weights $\{\w_1 , \ldots , \w_r \}$.  Let $I \subset \{1 , \ldots , r\}$, and consider the irreducible representation $\mu:\tilde\fg\ra\fgl(U)$ of highest weight $\l=\sum_{i  \in I} \l^{i }\o_{i }$.
Set $\fg=\mu(\tilde\fg)$, and let $\mu (G)\subset GL(U)$
be the associated Lie group so that
$ G/P \subset \BP U$ is the orbit of a highest
weight line.  Write $P=P_I\subset G$ for the parabolic subgroup obtained by deleting negative root spaces corresponding to roots having a nonzero coefficient on any of the simple roots $\a_{i }$, $i  \in I$.

Since $\tilde\fg$ is reductive, we have a splitting $\fgl(U)=\fg\op \fg\upperp$, where $\fg\upperp$ is the $\tilde\fg$-submodule of $\fgl(U)$ complementary to $\fg$. 
Let $\o\in \O^1(GL(U),\fgl(U))$ denote the Maurer-Cartan form of $GL(U)$, and let $\o_{\fg}$ and $\o_{\fg^\perp}$ denote the components of $\o$ taking values in $\fg$ and $\fg^\perp$, respectively.  

The bundle  $\cF^1_{G/P}$ 
admits a reduction to a bundle $\cF^G_{G/P} =\mu (G)$.  On this bundle  the Maurer-Cartan form pulls-back to take values in $\fg$,
that is,
$\w_{\fg^\perp} = 0$.  Conversely, all $\tdim (G)$ 
dimensional integral
manifolds of the system $I=\{\w_{\fg^\perp}\}$
are left translates of $\mu (G)$.
 
Let $Z=Z_I\subset \ft$ be the {\it grading element}
corresponding to $\sum_{i_s\in I}\a_{i_s}$. The grading element $Z_i$ for a simple root $\a_{i}$ has the  property that $Z_i(\a_j)=\d^i_j$. In general $Z=\sum_{i_s \in I} Z_{i_s }$. Thus, if   $(c^{-1})$ 
denotes the {\it inverse of the Cartan matrix}, then given a weight $\nu=\sum\nu^j\o_j$,  
\be\label{charelteqn}
Z(\nu)=\sum_{{1\le j \le r}\atop{i_s \in I}}\nu^j(c\inv)_{j,i_s} \, .
\ene
The grading element induces a $\bZ$-grading of $\fg = \oplus_{-k}^k \fg_k$.  To determine $k$ in the case $P=P_{\a_i}$ is a maximal parabolic, let  $\tilde\a=\sum m_j\a_j$ denote
the highest root, then $k=m_i$.

The module $U$ inherits a $\BZ$-grading 
$$
U=U_{Z(\l)}\op U_{Z(\l)-1}\op\cdots\op U_{Z(\l)-f} \, .
$$
The $U_j$ are eigen-spaces for $Z$.
This grading is compatible with the action of $\tilde \fg$: $\mu(\tilde\fg_i).U_j\subset U_{i+j}$. 
We   adopt the notational convention   of shifting the grading
on $U$  to begin
at zero. 
The component $U_0$ (formally named $U_{Z(\l)}$) is one dimensional, and
corresponds to the highest weight line of $U$, and $G \cdot \bP U_0 = G/P \subset \bP U$. (The 
labeling of the grading on $\fgl(U)=U^*\ot U$
is independent of our shift convention.)

Note, in particular, that the vector space 
$\hat T_{[\tId]}(G/P)/\hat{\tId}\simeq \fg/\fp$
is graded from $-1$ to $-k$.  The osculating grading on $U$ induces gradings of
$\fgl(U)$, $\fg$ and $\fg^\perp$.
In   Examples \ref{ograss} and \ref{e8p3} the summands in $T_xG/P$ appearing are in order from $-1$ to $-k$

We write
$$
\fgl(U)=\bigoplus_{s,j} \fgl(U)_{s,j}
$$
where the first index refers to the osculating grading (\S\ref{oscfiltsect})
induced by $G/P\subset \BP U$ and the second the root grading. We adopt the notational convention
$$
\fgl(U)_j=\bigoplus_{s } \fgl(U)_{s,j} \, ;
$$
so
if there is only one index, it refers to the root grading. Note that  the grading of $\fgl(U)$ is indexed by integers $-f\hd f$.   
\medskip

\subsection{Examples of tangent spaces and osculating filtrations of homogeneous varieties}

\begin{example} Consider $G(k,V)\subset \BP \La k V=\BP U$. Fix $E\in G(k,V)$. Then
the osculating sequence is
$$U_0=\La k E\subset (\La {k-1}E\ww V)\subset
(\La {k-2}E\ww \La 2V)\subset \cdots\subset (\La 1 E\ww \La{k-1}V)
\subset\La k V=U.
$$
\end{example}

\medskip

\noindent {\it Remark.}
The only nonzero Fubini forms of a homogeneously embedded CHSS are the fundamental forms.  For the adjoint varieties, the only nonzero
Fubini forms are $F_{2,2},F_{2,3},F_{2,4}$.
\medskip

One definition of $G/P\subset \BP V$ being cominuscule is that 
$T_{Id}(G/P)$ is an irreducible $P$-module. Here are some examples
describing  tangent spaces and osculating
sequences of non-cominuscule varieties.

\begin{example}\label{ograss}
For orthogonal Grassmannians $G_Q(k,V)\subset \BP \La k V$ (assume $k<\frac 12\tdim V$),
\begin{align*} (T_E(G_Q(k,V)))_{-1}&=E^*\ot (E\upperp/E)\\
  (T_E(G_Q(k,V)))_{-2}&=   \La 2 E^*,
\end{align*}
 where the $\perp$ refers
to the $Q$-orthogonal complement, $gr$ to the associated
graded vector space of the filtered vector space $T_EG_Q(k,V)$.
Note that $E\subset E\upperp$ because $E$ is isotropic.
\end{example}

\begin{example}\label{e8p3}
For the $89$ dimensional variety $(E_8/P_3) \subset \BP V_{\o_3}=\pp{6696999}$
\begin{align*}
T_{-1}&=U\ot \La 2 W\\
T_{-2}&=  \La 4 W\\
T_{-3}&= U\ot \La 6 W\\
T_{-4}&=  W 
\end{align*}
where $U=\BC^2$ is the standard representation of $A_1$
and $W=\BC^7$, the standard representation of $A_6$.
\end{example}

\medskip

\noindent {\it Remark.} For those familiar with Dynkin diagrams,
it is possible to obtain $T_{-1}$ and $T_{-f}$ pictorially, where
$T_{-f}$ is the last summand. For simplicity assume $P$ is maximal,
take the Dynkin diagram for $\fg$, delete the node
for $P$, and mark the adjacent nodes with the multiplicity of the
bond assuming an arrow points towards the marked note, otherwise
just mark with multiplicity one.

\begin{center}
\setlength{\unitlength}{3mm}
\begin{picture}(20,7)(13,.5)
\multiput(8,5)(2,0){5}{$\circ$}
\multiput(8.5,5.4)(2,0){4}{\line(1,0){1.6}} 
\put(16.4,5.5){\line(1,1){1.2}} 
\put(16.4,5){\line(1,-1){1.2}}
\put(17.6,6.5){$\circ$}  
\put(17.6,3.5){$\circ$}  
\put(8,2){$X=G_Q(4,12)$}
\put(14,5){$\bullet$}

\put(21,5){$\lra$}

\multiput(27,5)(2,0){2}{$\circ$}
\put(31,5){$\bullet$}
\multiput(27.5,5.4)(2,0){2}{\line(1,0){1.6}}
\put(34,5){$\bullet$}
\put(34.4,5.5){\line(1,1){1.2}} 
\put(34.4,5){\line(1,-1){1.2}}
\put(35.6,6.5){$\circ$}  
\put(35.6,3.5){$\circ$}  
\put(26,2){$T_{-1}=\BC^4\ot \La 2\BC^4$}
\end{picture}
\end{center}
 
The last filtrand, $T_{-f}$ is obtained by marking the node(s) associated
to the adjoint representation of $\fg$ and taking the dual module
in the new diagram. These coincide iff $G/P$ is CHSS.

\medskip

\subsubsection{Symplectic Grassmannians}
Here is the full osculating sequence and some details for
the symplectic Grassmannians taken from \cite{LM0}  (where many other cases with $P$ maximal
 may be found as well):

Let  $G_\o(k,2n)=C_n/P_k$ denote
  the Grassmanian of $k$-planes
isotropic for a symplectic form. Its minimal embedding is to 
$V_{\o_k} = \Lambda^{\langle k\rangle }V =
\Lambda^kV/(\Omega\we\Lambda^{k-2}V)$, 
 the $k$-th reduced exterior power of $V=\bcc{2n}$,
where $\Omega\in \La 2 V$ denotes the symplectic form on $V^*$ induced from $\o\in \La 2 V^*$. 

Let $E\in G_{\o}(k,V)$ and write $U=E\upperp/E$.  A straightforward computation 
shows that   $V_{\o_k}$   
has the following decomposition  as an  $H=SL(E)\times
Sp(U)$-module:
$$
V_{\o_k}= \Lambda^{\langle k\rangle }V =
\bigoplus_{a,b}\Lambda^bE^*\ot\Lambda^{a+b}E^*\ot\Lambda^{\langle
a\rangle}U.
$$
(Here $\tdim E=k$ and $\tdim U=2n-2k$.)
Note that $U$ is endowed with a symplectic form induced 
by the symplectic form on $V = \CC^{2n}$.

\begin{proposition}\cite{LM0} Let $E\in G_{\o}(k,2n)$, 
let $E\upperp\supset E$ denote the $\Omega$-orthogonal complement
to $E$ and let $U =
E^{\perp}/E$. Then the tangent  space and normal spaces of $G_{\o}(k,2n)$
are, as
$G_0$-modules,
\begin{align*}
T_{-1}  &  =    E^*\ot U,\\
T_{-2} &  =    S^2 E^*,\\
N_{2,-2} &  =    \Lambda^2 E^*\ot \Lambda^{\langle 2\rangle} U\\
 N_{2,-3} &  =   
S_{21} E^*\ot U\\
 N_{2,-4} &  =    S_{22}E^*, \\
N_{p,*} &  =    \bigoplus_{a+b+c = p}\Lambda^{\langle a\rangle}U\ot S_{2^{b-c}1^{a+2c}}E^* \\
&  =   \bigoplus_{d+e = p}\Lambda^dU\ot S_{2^e1^d}E^*.
\end{align*}
 $S_{\pi}E$ is the irreducible $GL(E)$ module associated to the partition $\pi$.
(Here $S_{2^a1^b}E$ corresponds to the partition with $a$ $2$'s and $b$ $1$'s.)
In particular, the length of the osculating sequence is equal to 
$k+1$, the last non zero term being $N_k\simeq\Lambda^k(\CC\op U)$.
\end{proposition}

\begin{corollary}\cite{LM0} 
$$
B (II_{ G_{\o}(k,2n), E})= 
\ppp \overline{\{e\ot u\oplus e^2 \mid e\in E^*\backslash \{0\},\;u\in U
\backslash \{0\}\} }. 
$$ 
This set of asymptotic directions contains an open   dense $P$-orbit, 
the boundary of which is the union of the two (disjoint) closed 
$H$-orbits 
$$Y_1\simeq \PP^{k-1}\times\PP^{2n-2k-1}\subset \PP(T_{-1})\quad
and \quad Y_2\simeq v_2(\PP^{k-1})\subset \PP(T_{-2}).$$
\end{corollary}

\medskip

\subsubsection{The bigrading for adjoint varieties}
For adjoint varieties, $T_xX^{ad}_G$ has a two step filtration, with
the hyperplane being the first filtrand, and the osculating
sequence is simply $\hat x\subset \hat T\subset U$.  
 The induced bi-grading on $\fgl(U)_{\textrm{osc},\textrm{alg}}$
is indicated in the table below.
\begin{center}\renewcommand{\arraystretch}{1.3}
\begin{tabular}{c||c|c|c|c|c|c}
    & $\hat x^*$ & $T_{-1}{}^*$ & $T_{-2}{}^*$ & $N_{-2}{}^*$ & $N_{-3}{}^*$ 
    & $N_{-4}{}^*$ \\  \hline \hline
 $\hat x$ & 
 (0,0) & (-1,1) & (-1,2) & (-2,2) & (-2,3) & (-2,4) \\ \hline
 $T_{-1}$ & 
 (1,-1) & (0,0) & (0,1) & (-1,1) & (-1,2) & (-1,3) \\ \hline
 $T_{-2}$ & 
 (1,-2) & (0,-1) & (0,0) & (-1,0) & (-1,1) & (-1,2) \\ \hline
 $N_{-2}$ & 
 (2,-2) & (1,-1) & (1,0) & (0,0) & (0,1) & (0,2) \\ \hline
 $N_{-3}$ & 
 (2,-3) & (1,-2) & (1,-1) & (0,-1) & (0,0) & (0,1) \\ \hline
 $N_{-4}$ & 
 (2,-4) & (1,-3) & (1,-2) & (0,-2) & (0,-1) & (0,0)
\end{tabular}
\end{center}
In all cases $T_{-1},T_{-2}$ may be determined by the remark above
($T_{-2}$ is the trivial module as there is no node left to mark),
and $N_{-2}=I_2(Y)^*$, $N_{-3}=I_3((\t(Y))_{sing})$ and $N_{-4}$ is
trivial (corresponding to the quartic generating $I_4(\t(Y))$).

\section{Lie algebra cohomology and Kostant's theory}\label{Liealgcohsect}
\label{kosapply}\label{kossect}\label{reconcilesect}

Let $\fl$ be a Lie algebra and let $\Gamma$ be an $\fl$-module.
Define maps
$$
\partial^{j}: \La j\fl^*\ot \Gamma \ra \La{j+1}\fl^*\ot \Gamma
$$
in the only natural way possible respecting the Leibniz rule.
This gives rise to a complex and we
  define $H^k(\fl,\Gamma):=\tker\partial^k/\tim \partial^{k-1}$.
We will only have need of $\partial^0$ and $\partial^1$ which are defined explicitly
as follows:
if $X\in \Gamma$ and $v,w\in \fl$, then
$$\partial^0(X)(v)=v.X, 
$$
 and 
  if $\a\ot X\in \La 1\fl^*\ot \Gamma$,  then
\begin{equation}\label{del1}
\partial^1(\a\ot X)(v\ww w)=
\a([v,w])X+\a(v)w.X - \a(w) v.X
\end{equation}

Now let  $\fl$ be a graded Lie algebra and $\Gamma$   a graded
$\fl$-module. The
chain complex and Lie
algebra cohomology groups inherit  gradings as well.
Explicitly,
$$
\partial^{1}_{d}: \oplus_{i}(\fl_{-i})^* \ot \Gamma_{d-i}
\ra \oplus_{j\leq m}(\fl_{-j})^*\ww (\fl_{-m})^*\ot \Gamma_{d-j-m} \, .
$$

\medskip
 
Kostant \cite{kostant} shows that under the following circumstances one can
compute $H^k(\fl,\Gamma)$ combinatorially:
\begin{enumerate}
\item $\fl=\fn\subset\fp\subset  \fg$ is the nilpotent subalgebra of
a parabolic subalgebra of
a semi-simple Lie algebra $\fg$.

\item $\Gamma$ is a $\fg$-module.
\end{enumerate}

Under these conditions, letting $\fg_0\subset \fp$ be the the
(reductive) Levi factor of $\fp$, $H^j(\fn,\Gamma)$ is naturally
a $\fg_0$-module. Kostant shows that for any irreducible module $\Gamma$
it is essentially trivial to compute $H^1(\fn,\Gamma)$, one just
examines certain simple reflections in the Weyl group. However,
in our situation, where we need to compute $H^1(\fg,\fg\upperp)$, there
may be numerous components to $\fg\upperp$, and moreover we would
like to avoid a case by case decomposition. Here the beauty of
the grading element comes in, because it is easy to prove that
in many situations $H^1_d(\fg,\Gamma)$ is zero in positive
degree.
This is well documented
in \cite{Y,HY,LRrigid} among other places.

\section{From the Fubini EDS to Filtered EDS}\label{filtersect}

I now explain how we were led to work with filtered EDS in an effort
to use Lie algebra cohomology to determine rigidity of homogeneous varieties.

\subsection{Problem 1: Osculating v.s. root gradings}\label{problem1}
As mentioned above, for CHSS, the osculating grading coincides with the root grading, but
for all other homogeneously embedded homogeneous varieties this fails.
Thus to have any hope to exploit Lie algebra cohomology, we need
to work with an EDS that respects the root grading.

From now on we will work on $SL(U)\subset GL(U)$ which will not
change anything regarding our study of rigidity of subvarieties of $\BP U$.
Define the $(I_{p},J_{p})$ system on
$SL(U)$ by
$I_p=\{ \o_{\gp{\leq p}}\}$, $J_p=\{ I_p, \o_{\fg_-}\}$.

In specific examples, after a short calculation, the $k$-th order Fubini system can be shown to be
strictly stronger than some $(I_p,J_p)$ system (where of course $p$ depends on $k$).
At the moment we have no general method of determining this, but we do so
uniformly for adjoint varieties in \cite{LRrigid}. In summary, this problem
is easy to resolve in specific cases or even classes of cases, but work
remains to resolve the general case.

Here is the proof in the adjoint case:

\begin{proposition}\label{prop:Fub3=>-1}
Every integral manifold of the third-order Fubini system $(I_{\tFub_3},J_{\tFub_3})$ for
a given adjoint variety is an integral
manifold of the $(I_{-1},J_{-1})$ system for the
same adjoint variety.
\end{proposition}

\begin{proof}
Suppose that $\cF \subset SL(U)$ is an integral manifold of third-order Fubini system.    We wish to show that the $\fg^\perp_{*,<0}$--valued component of the Maurer-Cartan form vanishes when pulled-back to $\cF$.  That the $\fg^\perp_{>0,*}$--valued component vanishes is an immediate consequence of the injectivity of the second fundamental form $F_2$ on each homogeneous component.

Referring to the table above, we see that there remain
four blocks of the component of the Maurer-Cartan form 
in $\gp{*,<0}$ to consider: the three $(0,-1)$ blocks $\w_{T_{-2}\ot T_{-1}^*}$, $\w_{N_{-3}\ot N_{-2}^*}$ and $\w_{N_{-4}\ot N_{-3}^*}$; and the singleton $(0,-2)$ block $\w_{N_{-4}\ot N_{-2}^*}$.  The third Fubini form is defined by (3.5) of \cite[\S3.5]{IvL}.  The vanishing of the $\fg^\perp$--component of the first two blocks is a consequence of the $S^3 T_{-1}^* \ot N_{-3}$ component of $F_3$.  (This is the only nonzero component of $F_3$.)  The vanishing of the $\fg^\perp$--component of the third and fourth blocks is given by the $S^3T^* \ot N_{-4}$--component of $F_3$.
\end{proof}

\subsection{Problem 2: Even the  systems defined by the root grading do not  
lead to Lie algebra cohomology}
\label{problem2}
For simplicity we take  $p=-1$ and $k=2$.
Notice that $\fg_{s}^\perp = \fsl(U)_{s}$ for all $s \le -3$.  Abbreviate
$$ \w_{\fsl(U)_s} \ =: \ \w_s \, , $$
so that $\w_{\fg^\perp_s} = \w_{s}$, for all $s \le -3$.  Thus  
\begin{displaymath}
  I_{-1} \ = \ \left\{ 
    \w_{\fg^\perp_{-1}}  \, , \ 
    \w_{\fg^\perp_{-2}} \, , \ 
    \w_{-3} \, , \ \ldots \, , \ \w_{-f} 
  \right\} \, .
\end{displaymath}

The calculations that follow utilize the Maurer-Cartan equation (see \S\ref{projframessect}), and
that  $[\fg,\fg] \subset \fg$ and $[\fg,\fg^\perp] \subset \fg^\perp$.  It is easy to see that $\td \w_{s} \equiv 0$ modulo $I_{-1}$ when $s \le -3$.  Next, computing modulo $I_{-1}$, 
$$
  -\td \, \o_{\gp{-2}}   \equiv  
  \left[  \o_{\fg_{-2}} , \w_{\fg^\perp_0} \right]
$$
and
$$
  -\td \o_{\gp{-1}}   \ \equiv \\   
  \left[ \w_{\fg_{-2}} , \w_{\fg^\perp_1} \right] \ + \ 
  \left[ \w_{\fg_{-1}} , \w_{\fg^\perp_0} \right]  .
 $$
 In order for these two equations to be satisfied, on an integral element we
must have
\begin{align}
\o_{\gp 0}&= \l_{0,1}(\o_{\fg_{-1}})+ \l_{0,2}(\o_{\fg_{-2}})\\
\o_{\gp 1}&= \l_{1,1}(\o_{\fg_{-1}})+ \l_{1,2}(\o_{\fg_{-2}})
\end{align}
for some $\l_{i,j}\in \gp i\ot \fg_{-j}^*$.

Consider the degree two homogeneous component $\d_2$ of the Spencer differential
$\d: A\ot V^*\ra W\ot \La 2 V^*$, where
$A=\gp 1\op \gp 0$, $W=\gp{\leq -1}$ (but we may and will   ignore the first
derived system $\gp{\leq -3}$) $V=\fg_{-1}\op \fg_{-2}$:
\begin{align*}
\d_2: (\gp 1\ot\fg_{-1}^*)\op  (\gp 0\ot\fg_{-2}^*)&\ra 
 (\gp{-1}\ot \fg_{-1}^*\ww \fg_{-2}^*)\op (\gp{-2}\ot \fg_{-2}^*\ww \fg_{-2}^*)
  \\
\l_{1,1}\op \l_{0,2}&\mapsto \{ (u_{-1}\ww v_{-2})
\mapsto [\l_{1,1}(u_{-1}),v_{-2}]+ [u_{-1},\l_{0,2}(v_{-2})] \\
&\ \ \ \op   (x_{-2},y_{-2})\mapsto
[\l_{0,2}(x_{-2}),y_{-2}]+ [x_{-2},\l_{0,2}(y_{-2})] \}
\end{align*}
Here $x_{-2}\in \fg_{-2}$ etc...
This is exactly the Lie algebra cohomology differential
$\partial^1_2$!
Now consider the degree one component $\d_1$
 \begin{align*}
\d_1: (\gp 0\ot\fg_{-1}^*)  &\ra 
 (\gp{-1}\ot \fg_{-1}^*\ww \fg_{-1}^*)\op (\gp{-2}\ot \fg_{-1}^*\ww \fg_{-2}^*)
  \\
\l_{0,1}&\mapsto \{ (u_{-1}\ww v_{-1})\op (x_{-1},y_{-2})
\mapsto [\l_{0,1}(u_{-1}),v_{-1}]+ [u_{-1},\l_{0,1}(v_{-1})]\op  [\l_{0,1}(x_{-1}),y_{-2}]\}
\end{align*}
This fails to be the Lie algebra cohomology differential because
we are \lq\lq missing\rq\rq\ a term $\l_{-1,2}([u_{-1},v_{-2}])$ on
the right hand side.
One can try to \lq\lq fix\rq\rq\ this by
adding in such a term. At first this appears unnatural, but if one
takes into account that there is a natural filtration
on our manifold, it is not unreasonable to weaken
the condition $\o_{\gp{-1}}=0$ to the
condition $\o_{\gp{-1}}|_{T_{-1}}=0$, i.e.,
$\o_{\gp{-1}}=\l_{-1,2}(\o_{\fg_{-2}})$ where $\l_{-1,2}\in \gp{-1}\ot \fg_{-2}^*$
at each point of our manifold.

We make this \lq\lq fix\rq\rq\ precise and natural with the introduction of filtered EDS:

\subsection{The Fix for problem 2: Filtered EDS}

\begin{definition} Let $\Sigma$ be a manifold equipped with a filtration of its
tangent bundle
$T^{-1}\subset T^{-2}\subset \cdots \subset T^{-f}=T\Sigma$.
Define an \emph{$r$-filtered} Pfaffian EDS on $\Sigma$ to be a filtered ideal
$I\subset T^*\Sigma$ whose
integral manifolds are the immersed submanifolds $i: M\ra \Sigma$
such that 
$i^*(I_{u})|_{i^*(T^{u-r})}=0$ for all $u$, with the convention that
$T^{-s}=T\Sigma$ when $-s\leq -f$.
\end{definition}

Another way to view filtered EDS is to consider the ordinary EDS on the sum of the bundles
$I_u\otimes (T\Sigma/T^{u+r})$. In our case these bundles will be trivial with
fixed vector spaces as models.

Define $(I^\textsf{f}_{p},\Omega)$ to be the $(p+1)$-filtered EDS on $GL(U)$ with filtered ideal
  $I^\textsf{f}_{p}:=\o_{\fg^{\perp}_{\leq p}}$ and independence condition
$\Omega$ given by the wedge product of the forms in $\o_{\fg_{-}}$.   We may view this as an ordinary
EDS on 
$$GL(U)\times \left( [\fg^{\perp}_{p}\ot (\fg_{-2}\op \cdots \op \fg_{-k})^*]
\oplus [\fg^{\perp}_{p-1}\ot (\fg_{-3}\op \cdots \op \fg_{-k})^*]
\oplus \cdots \oplus [\fg^{\perp}_{p-k+2}\ot  \fg_{-k}^*] \right)
$$
where, giving $\fg^{\perp}_i\ot \fg_{-j}^*$ linear coordinates $\l_{i,j}$, we
have
\begin{eqnarray} \label{psystem} 
  I^\textsf{f}_p & = &
  \big\{ \ \o_{\gp s} \, , \ s\leq p-k+1 \, ; \quad
  \o_{\gp{p-k+2}}-\l_{p-k+2,k}(\o_{\fg_{-k}}) \, , \\
  & & \nonumber \hspace{10pt}
  \o_{\gp{p-k+3}}-\l_{p-k+3,k}(\o_{\fg_{-k}})-\l_{p-k+3,k-1}(\o_{\fg_{-k+1}}) \, , \ \ldots \\
  & & \nonumber \hspace{10pt}
  \o_{\gp{p}}-\l_{p,k}(\o_{\fg_{-k}})-\cdots -\l_{p,2}(\o_{\fg_{-2}}) \ \big\}
\end{eqnarray}
However, as is explained  below, it is more natural to work in the 
category of filtered EDS.
\medskip

 Returning to the $p=-1$, $k=2$ system,
 the first derived system is $\o_{\leq -4}$.
Computing similarly to above, only now modulo $I^\textsf{f}_{-1}$, we obtain 
\begin{eqnarray}
  - \td \, \w_{-3} & \equiv & \label{eqn:-3}
  \left[ \w_{\fg_{-2}} , \w_{\fg^\perp_{-1}} \right] 
  \ \equiv \
  \left[ \w_{\fg_{-2}} , \lambda_{-1,2}(\w_{\fg_{-2}}) \right]  \, , \\
  -\td \, \o_{\gp{-2}} & \equiv & \label{eqn:-2}
  \left[\o_{\fg_{-2}} , \o_{\gp 0} \right] \ + \ 
  \left[\o_{\fg_{-1}},\w_{\gp {-1}} \right] \ + \
  \left[ \w_{\gp {-1}}, \w_{\gp {-1}} \right]_{\fg\upperp}\\
  & \equiv & \nonumber
  \left[\o_{\fg_{-2}} , \o_{\gp 0} \right] \ + \ 
  \left[\o_{\fg_{-1}},\lambda_{-1,2}(\w_{\fg_{-2}}) \right] \ + \
  \left[\lambda_{-1,2}(\w_{\fg_{-2}}) , 
        \lambda_{-1,2}(\w_{\fg_{-2}}) \right]_{\fg\upperp} \, ,
\end{eqnarray}
  
\begin{equation}\label{eqn:-1}
\renewcommand{\arraystretch}{1.3}
\begin{array}{l}
  -\td \left(\o_{\gp{-1}}-\l_{-1,2}(  \o_{\fg_{-2}})\right)  \ \equiv \\
  \hspace{80pt} 
  \left[ \w_{\fg_{-2}} , \w_{\fg^\perp_1} \right] \ + \ 
  \left[ \w_{\fg_{-1}} , \w_{\fg^\perp_0} \right] \ + \ 
  \left[ \w_{\fg_{-1}^\perp} , \w_{\fg_0} \right] \ + \ 
  \left[ \w_{\fg^\perp_{-1}} , \w_{\fg^\perp_{0}} \right]_{\fg^\perp} \\
  \hspace{80pt} + \ 
  \td \lambda_{-1,2} ( \wedge \w_{\fg_{-2}} ) \ - \
  \lambda_{-1,2} \left(
    \left[ \w_{\fg_{-2}} , \w_{\fg_0} \right] + 
    \left[ \w_{\fg_{_1}} , \w_{\fg_{-1}} \right] +
    \left[ \w_{\fg^\perp_{-1}} , \w_{\fg^\perp_{-1}} \right]_{\fg}
  \right) \\
  \hspace{65pt} \equiv  \
  \left[ \w_{\fg_{-2}} , \w_{\fg^\perp_1} \right] \ + \ 
  \left[ \w_{\fg_{-1}} , \w_{\fg^\perp_0} \right] \ + \ 
  \left[ \lambda_{-1,2}(\w_{\fg_{-2}}) , \w_{\fg_0} \right] \\
  \hspace{80pt} + \
  \left[ \lambda_{-1,2}(\w_{\fg_{-2}}) , 
         \w_{\fg^\perp_{0}} \right]_{\fg^\perp} \ + \ 
  \td \lambda_{-1,2} ( \wedge \w_{\fg_{-2}} )  \\
  \hspace{80pt} - \
  \lambda_{-1,2} \left(
    \left[ \w_{\fg_{-2}} , \w_{\fg_0} \right] + 
    \left[ \w_{\fg_{-1}} , \w_{\fg_{-1}} \right] +
    \left[ \lambda_{-1,2}(\w_{\fg_{-2}}) , \lambda_{-1,2}(\w_{\fg_{-2}}) \right]_{\fg}
  \right)
\end{array}
\end{equation}
Here   $[\cdot , \cdot]_{\fg}$ (resp. $[\cdot,\cdot]_{\fg\upperp}$) denotes
the component of the bracket taking values in $\fg$ (resp. $\fg\upperp$).

Note that if were were to
view the filtered EDS as an  ordinary EDS on $SL(U)\times \gp{-1}\ot \fg_{-2}^*$, the term
$\l_{-1,2}$ is part of the torsion whereas here  
  it is simply part of  the tableau of the   filtered  Spencer differential.

The degree one   homogeneous component of (\ref{eqn:-3},\ref{eqn:-2},\ref{eqn:-1}),
 is as follows:  ${\mathbf\l}_1:=\oplus_{s=-1}^0\l_{s,1-s}$ must be in the kernel of the map
$$
\d_1:
\oplus_{s=-1}^0(\gp{s}\ot \fg^*_{-s-1})
\ra   (\gp{-1}\ot  \fg^*_{-1}\ww \fg^*_{-1})
\op (\gp{-2}\ot \fg^*_{-1}\ww \fg^*_{-2})
$$
defined as follows.  Given $u_{-1},v_{-1}\in \fg_{-1}$, 
\begin{equation}\label{d1der}
\d_1({\mathbf\l}_1)(u_{-1}\ww v_{-1})
=[\l_{0,1}(u_{-1}),v_{-1}]+[u_{-1},\l_{0,1}(v_{-1})]-\l_{-1,2}([u_{-1},v_{-1}]).
\end{equation}
For $u_{-1}\in \fg_{-1},v_{-2}\in \fg_{-2}$
$$
\d_1({\mathbf\l}_1)(u_{-1}\ww v_{-2})
=[\l_{0,1}(u_{-1}),v_{-2}]+[u_{-1},\l_{-1,2}(v_{-2})] \, .
$$
That is,
$\d_1=\partial^1_1$, where $\partial^1_1$ is the Lie algebra cohomology differential 
described in \S\ref{Liealgcohsect}.

Moreover, $\fg^\perp_{1}  = \fn\cap  \fgl(U)_{ 1}$, and the Lie algebra cohomology denominator $\partial^0_{1}(\fg^\perp_{1})$ is the space of admissible normalizations of the prolongation coefficients ${\mathbf \l}_{1}$.  Thus, the vanishing of $H^1_{1}(\fg_-,\fg^\perp)$ implies that normalized integral manifolds of the $(I^f_{-1},\Omega)$ system are in one to one
correspondence with integral manifolds of the $(I^f_{0},\Omega)$ system.

Punch line: by working with filtered EDS and by homogeneous degree
we do obtain Lie algebra cohomology. The vanishing of the Lie algebra cohomology
reduces the system to the $(I_0^f,\Omega)$ system, and vanishing
of the Lie algebra cohomology group $H^1_2(\fg_{-},\fg\upperp)$
moves one to the $(I_1^f,\Omega)$ system etc...  Moreover,
there was nothing special about beginning with $p=-1$. The final result is:

\begin{theorem} \label{transthm}\label{mainthm}
Let $U$ be a complex vector space, and $\fg\subset \fgl(U)$ a represented complex semi-simple Lie algebra.  Let $Z = G/P \subset \bP U$ be the corresponding homogeneous variety  (the orbit of a highest weight line).  Denote the induced $\bZ$-gradings by $\fg = \fg_{-k} \op \cdots \op \fg_k$ and $U = U_0 \op \cdots \op U_{-f}$.  Fix an integer $p \ge -1$, and let $(I_p^\textsf{f},\Omega)$  denote the linear Pfaffian system given by \eqref{psystem}.
If $H^1_d(\fg_{-},\fg\upperp)=0$, for all $d\geq p+2$, then the homogenous variety $G/P$ is rigid for the $(I_p^\textsf{f},\Omega)$ system.
\end{theorem}

\section{Open questions and problems}
\begin{itemize}
\item Does $H^1_d(\fg_-,\fg\upperp)$ nonzero imply flexibility? If so, can one prove
this directly and in general without going through the (sometimes quite long) 
Cartan algorithm?

\item Give a uniform description of the $F_k$ for all $G/P$'s to obtain uniform
determinations of Fubini rigidity.

\item Determine the class of extrinsically realizable non flat parabolic
geometries modeled on $Flag_{1,2}(\BC^3)$ as some natural class of
parabolic geometries.

\item Apply Cap's machinery to study parabolic geometries having families
of differential operators whose kernel is large but not maximal.

\end{itemize}


\end{document}